\documentclass[12pt]{amsart}
\usepackage{verbatim, latexsym, amssymb, amsmath,color, mathabx}
\usepackage[shortlabels]{enumitem}
\usepackage{epsfig}
\usepackage{amsfonts}
\usepackage{mathtools}
\usepackage{geometry}
\usepackage{hyperref}

\newtheorem{thm}{Theorem}[section]
\newtheorem{prop}[thm]{Proposition}
\newtheorem{lem}[thm]{Lemma}
\newtheorem{defn}[thm]{Definition}
\newtheorem{cor}[thm]{Corollary}
\newtheorem{rem}[thm]{Remark}
\newtheorem{question}[thm]{Question}

\DeclareMathOperator{\Diff}{\operatorname{Diff}}
\DeclareMathOperator{\Ham}{\operatorname{Ham}}

\DeclareMathOperator{\Cal}{\operatorname{Cal}}

\DeclareMathOperator{\bR}{\mathbb{R}}
\DeclareMathOperator{\bZ}{\mathbb{Z}}

\DeclareMathOperator{\cD}{\mathcal{D}}
\DeclareMathOperator{\cH}{\mathcal{H}}
\DeclareMathOperator{\cO}{\mathcal{O}}
\DeclareMathOperator{\cP}{\mathcal{P}}

\begin{document}
\title[Generic equidistribution]{Generic equidistribution for area-preserving diffeomorphisms of compact surfaces with boundary}
\author{Abror Pirnapasov}
\author{Rohil Prasad}

\maketitle

\begin{abstract}
We prove that a generic area-preserving diffeomorphism of a compact surface with non-empty boundary has an equidistributed set of periodic orbits. This implies that such a diffeomorphism has a dense set of periodic points, although we also give a self-contained proof of this ``generic density'' theorem. One application of our results is the extension of mean action inequalities proved by Hutchings and Weiler for the disk and annulus to generic Hamiltonian diffeomorphisms of any compact surface with boundary. 
\end{abstract}

\section{Introduction}

\subsection{Statement of main results} Fix a smooth, compact, oriented surface $Z$ with smooth boundary $\partial Z$, and also fix a choice of smooth area form $\omega$ on $Z$. Write $\Diff(Z, \omega)$ for the space of diffeomorphisms of $Z$ which preserve the area form $\omega$. Any $\phi \in \Diff(Z, \omega)$ arising as the time-one flow of a time-dependent Hamiltonian $H \in C^\infty(\bR/\bZ \times Z)$, such that $H(t, -)$ is locally constant on $\partial Z$ for each $t$, is called a Hamiltonian diffeomorphism\footnote{We stress that these are not compactly supported, as much of the literature assumes, and that it is important for our purposes that they are not compactly supported.}, and the space of Hamiltonian diffeomorphisms forms a subgroup $\Ham(Z, \omega) \subset \Diff(Z, \omega)$. 

In the case where $Z$ is closed ($\partial Z = \emptyset$), it was recently shown that a Baire-generic element of $\Diff(Z, \omega)$ has a dense set of periodic points \cite{CGPZ21, edtmairHutchings}. This was also quantitatively refined in \cite{equidistribution}. The main goal of this paper is to prove these results in the case where the boundary $\partial Z$ is non-empty, which we will assume from now on. Beyond the intrinsic interest of the $C^\infty$-generic density/equidistribution theorem\footnote{See the Bourbaki seminar \cite{Humiliere21} by Humili\`{e}re for a historical account of work on the generic density theorem, which goes back to the work of Pugh in the $60$s.}, our results have further applications in low-dimensional conservative smooth dynamics. We provide such an application in Section~\ref{subsec:mean_action} below. Also, since the first version of this paper appeared, our results were used by Enciso--Peralta-Salas \cite{EPS23} to obstruct the topological realization of divergence-free vector fields on $\bR^3$ as magnetohydrostatic equilibria. 

\subsubsection{Generic density}\label{subsec:intro_density}

We begin by stating our ``generic density'' results. These are special cases of the quantitative ``generic equidistribution'' results below, but they are simpler to state and prove, so we include them as separate results. 

\begin{thm} \label{thm:boundary_density}
Let $Z$ be a smooth, compact, oriented surface with smooth, non-empty boundary $\partial Z$, and let $\omega$ be a smooth area form on $Z$. Then a $C^\infty$-generic element of $\Diff(Z, \omega)$ has a dense set of periodic points.
\end{thm}

\begin{rem}
Theorem~\ref{thm:boundary_density} and the subsequent Theorems~\ref{thm:boundary_monotone_density}, \ref{thm:boundary_equidistribution}, and \ref{thm:boundary_monotone_equidistribution} do not assume that diffeomorphisms are compactly supported in the interior of $Z$; perturbations are required all the way up to the boundary. The space $\Diff_c(Z, \omega)$ of compactly supported area-preserving diffeomorphisms, unlike $\Diff(Z, \omega)$, is not a Baire space, so the designation of ``Baire-generic'' may be vacuous in this setting.
\end{rem}

In the closed case ($\partial Z = \emptyset$), it was observed that the generic density property holds within rational Hamiltonian isotopy classes in $\Diff(Z, \omega)$. ``Rationality'' in this setting is a certain homological property of area-preserving maps. We define it in \S\ref{subsec:mapping_torii}. The rationality condition holds for a dense subset of of $\Diff(Z, \omega)$. Any Hamiltonian diffeomorphism is rational, and more generally the property of rationality is preserved under Hamiltonian isotopy. We define a rational Hamiltonian isotopy class to be a Hamiltonian isotopy class in which every representative in $\Diff(Z, \omega)$ is rational. The next theorem extends the aforementioned observation to the case of surfaces with boundary. 

\begin{thm} \label{thm:boundary_monotone_density}
Let $Z$ be a smooth, compact, oriented surface with smooth, non-empty boundary $\partial Z$, and let $\omega$ be a smooth area form on $Z$. Then a $C^\infty$-generic element of a rational Hamiltonian isotopy class in $\Diff(Z, \omega)$ has a dense set of periodic points. 
\end{thm}

The subgroup $\Ham(Z, \omega)$ is a rational Hamiltonian isotopy class. Therefore, Theorem~\ref{thm:boundary_monotone_density} applies to $\Ham(Z, \omega)$. We conclude that the generic density result proved by Asaoka--Irie \cite{asaokairie} for Hamiltonian diffeomorphisms of closed surfaces also holds in the case of surfaces with boundary. 

\begin{cor}\label{cor:hamiltonian_density}
Let $Z$ be a smooth, compact, oriented surface with smooth, non-empty boundary $\partial Z$, and let $\omega$ be a smooth area form on $Z$. Then a $C^\infty$-generic Hamiltonian diffeomorphism of $(Z, \omega)$ has a dense set of periodic points. 
\end{cor}

\subsubsection{Generic equidistribution}

Theorem~\ref{thm:boundary_equidistribution} and \ref{thm:boundary_monotone_equidistribution} below are quantitative refinements of Theorems \ref{thm:boundary_density} and \ref{thm:boundary_monotone_density}. They assert that a generic area-preserving diffeomorphism of $Z$ (resp. generic element of a rational Hamiltonian isotopy class) admits a set of periodic orbits which equidistribute in some manner with respect to the area form $\omega$. Before providing statements, we establish some definitions and notation, including a precise definition of the notion of equidistribution of periodic orbits used in this paper. 

Fix any $\phi \in \Diff(Z, \omega)$. A \textbf{periodic orbit} of $\phi$ is a finite ordered multiset $S = \{x_1, \ldots, x_d\}$ of points in $\Sigma$ such that $\phi(x_i) = x_{i+1}$ for $i = 1, \ldots, d - 1$ and $\phi(x_d) = x_1$. The cardinality $|S| := d$ of $S$ is called the \textbf{period} of $S$. A periodic orbit $S = \{x_1, \ldots, x_d\}$ is \textbf{simple} if all of the $x_i$ are pairwise distinct. Use $\cP(\phi)$ to denote the set of all simple periodic orbits of $\phi$. 

A \textbf{orbit set}, the collection of which we denote by $\cP_{\bR}(\phi)$, is a formal finite linear combination of elements of $\cP(\phi)$ (simple periodic orbits) with positive real coefficients. An \textbf{integral orbit set}, the collection of which we denote by $\cP_{\bZ}(\phi)$, is an element of $\cP_{\bR}(\phi)$ where all the coefficients are positive integers. For any orbit set
$$\cO = \sum_{k=1}^N a_k \cdot S_k \in \cP_{\bR}(\phi)\quad\text{define}\quad|\cO| := \sum_{k=1}^N a_k \cdot |S_k| \in \bR.$$

Any periodic orbit $S = \{x_1, \ldots, x_d\}$ defines a continuous real-valued functional on the space $C^0(Z)$ of continuous functions by summing up the function on the points in $S$:
$$S(f) := \sum_{i=1}^d f(x_i).$$

Any orbit set $\cO$ defines a functional on $C^0(Z)$ via linear extension of the above formula: 
$$\cO(f) := \sum_{k=1}^N a_k \cdot S_k(f).$$

An \textbf{equidistributed sequence of orbit sets} is a sequence $(\cO_i)_{i \in \mathbb{N}}$ of orbit sets such that for any $f \in C^0(Z)$, the averages of $f$ on the orbit sets $\cO_i$ limit to the average of $f$ on $Z$:
$$\lim_{i \to \infty} \frac{\cO_i(f)}{|\cO_i|} = \big(\int_Z \omega \big)^{-1} \cdot \int_Z f\,\omega.$$

The most concise formulation of this statement is that $\{\cO_i/|\cO_i|\}_{i \in \mathbb{N}}$, considered as a sequence of Borel probability measures, converges weakly to $\omega$. We are now ready to state our ``generic equidistribution'' theorems. 

\begin{thm} \label{thm:boundary_equidistribution}
Let $Z$ be a smooth, compact, oriented surface with smooth, non-empty boundary $\partial Z$, and let $\omega$ be a smooth area form on $Z$. A $C^\infty$-generic element of $\Diff(Z, \omega)$ has an equidistributed sequence of orbit sets.
\end{thm} 

The next theorem is a ``generic equidistribution'' result for rational Hamiltonian isotopy classes in $\Diff(Z, \omega)$, a particular case of which are the Hamiltonian diffeomorphisms. See the discussion below Theorem~\ref{thm:boundary_density} above and \S\ref{subsec:mapping_torii} for more details. 

\begin{thm} \label{thm:boundary_monotone_equidistribution}
Let $Z$ be a smooth, compact, oriented surface with smooth, non-empty boundary $\partial Z$, and let $\omega$ be a smooth area form on $Z$. A $C^\infty$-generic element of a rational Hamiltonian isotopy class in $\Diff(Z, \omega)$ has an equidistributed sequence of orbit sets. 
\end{thm} 

Theorem~\ref{thm:boundary_monotone_equidistribution} implies the following quantitative refinement of Corollary~\ref{cor:hamiltonian_density}. 

\begin{cor}\label{cor:hamiltonian_equidistribution}
Let $Z$ be a smooth, compact, oriented surface with smooth, non-empty boundary $\partial Z$, and let $\omega$ be a smooth area form on $Z$. A $C^\infty$-generic Hamiltonian diffeomorphism of $(Z, \omega)$ has an equidistributed sequence of orbit sets. 
\end{cor}

\subsection{An application to mean action inequalities for surface maps} \label{subsec:mean_action}
We use Theorem~\ref{thm:boundary_monotone_equidistribution} to extend some quantitative dynamical results by Hutchings \cite{hutchings_mean_action} for area-preserving disk maps and Weiler \cite{weiler_mean_action} for area-preserving annulus maps to generic Hamiltonian diffeomorphisms of any compact surface with boundary. Let $Z$ be a smooth, compact, oriented surface with smooth, non-empty boundary $\partial Z$, and let $\omega$ be a smooth area form on $Z$. We fix a connected component $\gamma$ of $\partial Z$, and a primitive $\beta$ of $\omega$. Fix any $\phi \in \Diff(Z, \omega)$ such that the closed one-form $\phi^{*}\beta-\beta$ is exact. If $\phi$ is Hamiltonian, this condition holds for any primitive $\beta$. Let $f$ be any primitive of $\phi^*\beta - \beta$. This is well-defined up to the addition of a constant function, and we normalize it as follows. The ergodic average 
$$f^\infty(x) := \lim_{n \to \infty} \frac{1}{n}\sum_{k=0}^{n-1} f(\phi^k(x))$$
is a $\phi$-invariant measurable function which is integrable with respect to $\omega$. It is well-defined on a full-measure subset of $x \in Z$ containing the set of periodic points. It is also well-defined for every $x \in \gamma$ and is equal to a constant everywhere on $\gamma$. We denote by $f_{\phi, \beta, \gamma}$ the unique primitive of $\phi^*\beta - \beta$ such that the ergodic average $f^\infty_{\phi, \beta, \gamma}$ is equal to $0$ on $\gamma$. The function $f_{\phi, \beta, \gamma}$ and its ergodic average $f^\infty_{\phi, \beta, \gamma}$ are respectively called the \textbf{action} and the \textbf{asymptotic mean action} with respect to $\phi$, $\beta$, and $\gamma$. The average of the action is some kind of invariant of $\phi$, whose significance we will describe shortly. 

\begin{defn}[Calabi invariant]
    Fix $\phi \in \Diff(Z, \omega)$, a boundary component $\gamma$ of $Z$, and a primitive $\beta$ of $\omega$ such that $\phi^*\beta - \beta$ is exact. The \textbf{Calabi invariant} of $\phi$ with respect to $\beta$ and $\gamma$ is the average of $f_{\phi, \beta, \gamma}$:
    $$\Cal(\phi,\beta,\gamma):=\big(\int_Z \omega\big)^{-1} \cdot \int_{Z}f_{\phi,\beta,\gamma}\,\omega.$$
\end{defn}

The following lemma lists some basic properties of the asymptotic mean action. 

\begin{lem}\label{lem:propertymean}
Fix $\phi \in \Diff(Z, \omega)$. Assume that $\phi$ fixes a boundary component $\gamma$ of $Z$ and that there exists a primitive $\beta$ of $\omega$ such that $\phi^*\beta - \beta$ is exact. Then the asymptotic mean action satisfies the following properties:
\begin{enumerate}[(a)]
\item Let $\lambda$ be any primitive of $\omega$. If the closed one-form $\lambda - \beta$ is exact, then $f^\infty_{\phi, \lambda, \gamma} = f^\infty_{\phi, \beta, \gamma}$ almost everywhere.  
\item $\int_Z f^\infty_{\phi, \beta, \gamma}\,\omega = \int_Z f_{\phi, \beta, \gamma}\,\omega$. 
\item Fix a periodic orbit $S \in \cP(\phi)$ and any point $x \in S$. Then $f^\infty_{\phi, \beta, \gamma}(x)=S(f_{\phi, \beta, \gamma})/|S|$.
\end{enumerate} 
\end{lem}

\begin{proof}
 Lemma~\ref{lem:propertymean}(a) is proved by direct computation. Fix a primitive $g$ of $\lambda$. The ergodic average of the function $g \circ \phi - g$ is defined everywhere and identically equal $0$, so it follows by definition that $f^\infty_{\phi, \lambda, \gamma} = f^\infty_{\phi, \beta, \gamma}$ almost everywhere. Lemma~\ref{lem:propertymean}(b) follows from the Birkhoff ergodic theorem.  Lemma~\ref{lem:propertymean}(c) follows from the fact that, since the asymptotic mean action is $\phi$-invariant, $f^\infty_{\phi, \beta, \gamma}(x_i) = f^\infty_{\phi, \beta, \gamma}(x)$ for each $x_i \in S$. 
\end{proof}

The second property in Lemma~\ref{lem:propertymean} gives some indication of the significance of the Calabi invariant. The asymptotic mean action $f^\infty_{\phi, \beta, \gamma}$ at a point $x \in Z$ should be thought of as some kind of ``rotation number'' of $\phi$ at the point $x$. From this viewpoint, the Calabi invariant is the ``average rotation number'' of the map $\phi$. Various notions of rotation numbers play prominent roles in two-dimensional conservative dynamics. For example, a classical result of Mather \cite{mather} for twist maps of the annulus with boundary rotation numbers $r_- < r_+$ produces for any $a \in [r_-, r_+]$ a quasiperiodic point with rotation number $a$, which is a periodic point when $a$ is rational. It is interesting and potentially fruitful to consider whether the asymptotic mean action can be used to detect (quasi)periodic points in a similar manner, but for area-preserving diffeomorphisms on any surface, without any twist condition. The following question, which is already difficult due to the generality in which it is posed, makes a first step towards this goal. It asks whether periodic points can be found which have asymptotic mean actions above and below the Calabi invariant. Since $f^\infty_{\phi, \beta, \gamma}$ is $\phi$-invariant, this is equivalently formulated in terms of averages of the asymptotic mean actions over periodic orbits. 

\begin{question}\label{question} Fix $\phi \in \Diff(Z, \omega)$, a boundary component $\gamma$ of $Z$, and a primitive $\beta$ of $\omega$ such that $\phi^*\beta - \beta$ is exact. Then does the inequality
$$\inf_{S \in \mathcal{P}(\phi)} S(f^\infty_{\phi, \beta, \gamma})/|S| \leq \Cal(\phi, \beta, \gamma) \leq \sup_{S \in \mathcal{P}(\phi)} S(f^\infty_{\phi, \beta, \gamma})/|S|$$
hold?  
\end{question}

The first result concerning Question~\ref{question}, proved by Hutchings \cite{hutchings_mean_action}, is the following:
 
\begin{thm}\label{thm:Hutchings} (Hutchings, \cite{hutchings_mean_action}) Let $(\mathbb{D}, \pi^{-1}dx \wedge dy)$ denote the standard unit disk, equipped with the standard area form of area $1$. Fix the primitive $\beta = (2\pi)^{-1}(x dy - y dx)$. Then for any $\phi \in \Diff(\mathbb{D}, \omega)$, such that $\phi$ is a rotation near the boundary and $\Cal(\phi,\beta,\partial \mathbb{\mathbb{D}})<0$, we have the inequality
$$\inf_{S\in \cP(\phi)} S(f^\infty_{\phi,\beta,\partial \mathbb{D}})/|S| \leq \Cal(\phi,\beta,\partial\mathbb{D}).$$
\end{thm} 

Weiler \cite{weiler_mean_action} later proved a version of Theorem \ref{thm:Hutchings} with similar assumptions for area-preserving annulus diffeomorphisms. The first author \cite{Pirnapasov21} showed that Theorem~\ref{thm:Hutchings} holds without requiring the disk map to be a rotation near the boundary. As an application, it was proved that if $\phi$ is a pseudo-rotation then $\Cal(\phi,\beta,\partial \mathbb{D})=0$. An analogous result has been shown for three-dimensional Reeb flows with two simple periodic orbits in \cite{cristofarogardiner2021contact}. Le Calvez \cite{Cal} recently gave another proof of Theorem \ref{thm:Hutchings} using finite-dimensional methods. 

Our Theorem~\ref{thm:boundary_monotone_equidistribution} answers Question~\ref{question} in the affirmative for generic Hamiltonian diffeomorphisms of a compact surface with any genus and number of boundary components. 

\begin{thm}\label{thm:mean_action}
Let $Z$ be a compact surface with non-empty boundary and let $\omega$ be any area form. Fix any primitive $\beta$ of $\omega$ and any component $\gamma$ of $\partial Z$. Then for a $C^\infty$-generic $\phi \in \Ham(Z, \omega)$, the following inequality holds:
\begin{equation} \label{eq:calabi_inequality} \inf_{S \in \mathcal{P}(\phi)} S(f^\infty_{\phi, \beta, \gamma})/|S| \leq \Cal(\phi, \beta, \gamma) \leq \sup_{S \in \mathcal{P}(\phi)} S(f^\infty_{\phi, \beta, \gamma})/|S|.\end{equation}
\end{thm}

\begin{proof}
By the third item in Lemma \ref{lem:propertymean}, it suffices to prove the version of \eqref{eq:calabi_inequality} with the measurable function $f^\infty_{\phi, \beta, \gamma}$ replaced by the smooth function $f_{\phi, \beta, \gamma}$. By Theorem~\ref{thm:boundary_monotone_equidistribution}, a generic $\phi \in \Ham(Z, \omega)$ has a sequence of orbit sets $\cO_i$ such that 
$$\lim_{i \to \infty} \cO_i(f_{\phi, \beta, \gamma})/|\cO_i| = (\int_Z \omega)^{-1} \cdot \int_Z f_{\phi, \beta, \gamma}\,\omega = \Cal(\phi, \beta, \gamma).$$

For each $i$, choose the simple orbits $S_i^+$, $S_i^-$ from $\cO_i$ on which $f_{\phi, \beta, \gamma}$ has respectively the largest and smallest averages out of all orbits in $\cO_i$. The equidistribution implies
$$\liminf_{i \to \infty} S_i^-(f_{\phi, \beta, \gamma})/|S_i^-| \leq \Cal(\phi, \beta, \gamma) \leq \limsup_{i \to \infty} S_i^+(f_{\phi, \beta, \gamma})/|S_i^+|$$
which implies \eqref{eq:calabi_inequality}. 
\end{proof}

We also deduce the following refinement of Theorem~\ref{thm:mean_action}. It is analogous to a result of Bechara Senior--Hryniewicz-Salom\~ao \cite[Theorem $1.9$]{action_linking} for three-dimensional contact forms admitting equidistributed sequences of Reeb orbits. The proof is also very similar and we omit it. 

\begin{thm}\label{thm:mean_action_refined}
Let $Z$ be a compact surface with non-empty boundary and $\omega$ any area form of area $1$. Fix any primitive $\beta$ of $\omega$ and any component $\gamma$ of $\partial Z$. Then for a $C^\infty$-generic Hamiltonian diffeomorphism $\phi$ and any $\epsilon > 0$, the following is true. Suppose that $\Cal(\phi, \beta, \gamma) \geq 0$ (resp. $\Cal(\phi, \beta, \gamma) < 0$). Write $P_\epsilon^+$ for the closure of the set of periodic points $x$ such that $$f^\infty_{\phi, \beta, \gamma}(x) \geq (1 - \epsilon)\Cal(\phi, \beta, \gamma)$$ (resp. $\leq (1 - \epsilon)\Cal(\phi, \beta, \gamma)$) and $P_\epsilon^-$ for the closure of the set of periodic points such that 
$$f^\infty_{\phi, \beta, \gamma}(x) \leq (1 + \epsilon)\Cal(\phi, \beta, \gamma)$$ 
(resp. $\geq (1 + \epsilon)\Cal(\phi, \beta, \gamma)$). Then both $P_\epsilon^-$ and $P_\epsilon^+$ have positive measure with respect to the area form $\omega$. 
\end{thm}

Versions of Theorems~\ref{thm:mean_action} and \ref{thm:mean_action_refined} can be formulated for generic maps within any rational Hamiltonian isotopy class and any corresponding primitive $\beta$ for which the Calabi invariant can be defined. 

\subsection{Outline of proofs} 

\subsubsection{Generic equidistribution theorems} There are two main ideas behind the proof of Theorem \ref{thm:boundary_equidistribution}. We give an outline of the proof and highlight these ideas where they appear in the process. Fix a compact surface $Z$ with non-empty boundary and an area form $\omega$. We attach disks to the boundary components of $Z$ and extend the area form to produce a smooth, closed surface $\Sigma$ with area form $\Omega$; the boundary components form a set $L$ of disjoint embedded loops in $\Sigma$. 

The goal is, given $\phi \in \Diff(Z, \omega)$, to find a $C^\infty$-small perturbation $\phi'$ which has an equidistributed sequence of orbit sets. The first main idea is that, if the capping disks have equal areas, any $\phi \in \Diff(Z, \omega)$ can be extended to some $\psi \in \Diff(\Sigma, \Omega)$ which preserves $L$. This is done in \S\ref{subsec:extension}. This is useful because we can now work in the setting of closed surfaces, which is more amenable to the techniques from Periodic Floer homology featuring prominently in \cite{equidistribution}. We then adapt the arguments of \cite{equidistribution} to prove a technical ``near-equidistribution'' result for area-preserving maps of $\Sigma$ which preserve $L$. This is done in \S\ref{subsec:lag_equidistribution}; some technical transversality arguments are required in the proof. Applying this result shows that there is some $\psi' \in \Diff(\Sigma, \Omega)$ which is $C^\infty$-close to $\psi$ and has a ``nearly equidistributed'' orbit set, that is an orbit set whose distribution in the surface is almost uniform in a precise quantitative sense. The second main idea is that, since $\psi'$ preserves $L$, the loops $L$ separate the dynamics of $\psi'$ on $Z$ from the dynamics of $\psi'$ on $\Sigma \setminus Z$. The perturbation $\psi'$ restricts to a map $\phi' \in \Diff(Z, \omega)$ which is $C^\infty$-close to $\phi$. Moreover, if $\cO \in \cP_{\bR}(\psi')$ is the previously found nearly equidistributed orbit set, then a bit of careful analysis implies that $\cO \cap Z$ is an orbit set of $\phi'$ which is nearly equidistributed in $Z$. This is carried out in \S\ref{subsec:equidistribution_proof}. A standard Baire category argument then implies Theorem~\ref{thm:boundary_equidistribution}.

Theorem~\ref{thm:boundary_monotone_equidistribution} is proved in the same manner as Theorem~\ref{thm:boundary_equidistribution}. The same argument works because the extension construction in \S\ref{subsec:extension} preserves the rationality property if the areas of $\Sigma$ and $Z$ are rationally dependent, which is easy to ensure, and because if $\psi' \in \Diff(\Sigma, \Omega)$ is a rational map which preserves $L$, its restriction $\phi' \in \Diff(Z, \omega)$ is also rational. 

\subsubsection{Generic density theorems} Theorems~\ref{thm:boundary_density} and \ref{thm:boundary_monotone_density} are, as mentioned before, easier to prove than the generic equidistribution results. We outline the proof of Theorem~\ref{thm:boundary_monotone_density}; the proof of Theorem~\ref{thm:boundary_density} is similar. 

Start with some rational $\phi \in \Diff(Z, \omega)$ and extend it to a rational map $\psi \in \Diff(\Sigma, \Omega)$. Given any open set $U \subset \Sigma$, the results of \cite{CGPZ21} imply that it suffices to perturb $\psi$ by a small Hamiltonian diffeomorphism, compactly supported in $U$, to create a periodic point in $U$. Carrying this out for an open subset $U \subset Z \setminus \partial Z$ and restricting to $Z$ produces a $C^\infty$-small Hamiltonian perturbation $\phi'$ of $\phi$ which has a periodic point through $U$. Note that the rationality property is preserved under Hamiltonian perturbation, so $\phi'$ is also rational. This shows that a dense subset of $\Diff(Z, \omega)$ has a periodic point in $U$, after which a standard Baire category argument proves Theorem~\ref{thm:boundary_monotone_density}. 

\subsection{Acknowledgements} Many thanks to Dan Cristofaro-Gardiner for introducing the authors to each other and for useful discussions regarding this project. AP is grateful to Barney Bramham and Alberto Abbondandolo for helpful discussions. The work of AP is partially supported by the DFG SFB/TRR 191 ``Symplectic Structures in Geometry, Algebra and Dynamics'', Projektnummer 281071066-TRR 191, and the ANR CoSyDy ``Conformally Symplectic Dynamics beyond symplectic dynamics'' (ANR-CE40-0014). The work of RP is supported by the National Science Foundation under Award No. DGE-1656466. 

\section{Preliminaries}\label{sec:prelim}

This section introduces all of the necessary preliminaries for the arguments in this paper, besides those stated in the introduction. 

\subsection{Area-preserving and Hamiltonian diffeomorphisms}\label{subsec:maps}

Let $\Sigma$ be any compact, oriented, smooth surface, possibly with boundary, equipped with an area form $\Omega$.

\subsubsection{Hamiltonian diffeomorphisms of compact surfaces} Write $C^\infty(\bR/\bZ \times \Sigma)$ for the space of smooth functions on $\bR/\bZ \times \Sigma$. For any $H \in C^\infty(\bR/\bZ \times \Sigma)$ and $t \in \bR/\bZ$, we will use $H_t$ to denote its restriction to $\{t\} \times \Sigma$; this is a smooth function on $\Sigma$. If $\Sigma$ has non-empty boundary, we require $H_t$ to be locally constant on $\partial\Sigma$ for every $t$. This guarantees that the Hamiltonian vector field $X_{H_t}$ is tangent to the boundary, so its flow, denoted by $\{\psi^t_H\}_{t \in \bR}$, is defined for all times. 

For any $\phi \in \Diff(\Sigma, \Omega)$ and $H \in C^\infty(\bR/\bZ \times \Sigma)$, we write $\phi^H \in \Diff(\Sigma, \Omega)$ to denote the composition $\phi \circ \psi^1_H$ of $\phi$ with the time-one map of $H$. 

\subsubsection{Maps preserving Lagrangians} An \textbf{Lagrangian} $L$ is a set of pairwise disjoint, embedded closed curves on $\Sigma$. A Lagrangian $L = \sqcup_{i=1}^n \gamma_i$ is \textbf{inessential} if there is a set $\{D_i\}_{i=1}^n$ of pairwise disjoint embedded closed disks such that the component $\gamma_i$ is the boundary of $D_i$ for every $i = 1, \ldots, n$. 

Let $L$ be any Lagrangian in $\Sigma$ and write $C^\infty(\Sigma; L)$ for the space of smooth functions on $\Sigma$ which are locally constant on $L$. We write $\Diff(\Sigma, \Omega; L)$ for the space of area-preserving maps $\phi \in \Diff(\Sigma, \Omega)$ such that $\phi(L) = L$. We write $C^\infty(\bR/\bZ \times \Sigma; L)$ for the space of smooth functions $H$ such that $H(t, -) \in C^\infty(\Sigma; L)$ for every $t \in \bR/\bZ$.

\subsubsection{Families of Hamiltonians} For any $N \geq 1$, write $C^\infty([0,1]^N \times \bR/\bZ \times \Sigma)$ for the space of smooth functions on $[0,1]^N \times \bR/\bZ \times \Sigma$. For any $N \geq 1$, $H \in C^\infty([0,1]^N \times \bR/\bZ \times \Sigma)$ and $\tau \in [0,1]^N$, we use $H^\tau$ to denote the function $H(\tau, -) \in C^\infty(\bR/\bZ \times \Sigma)$. For any Lagrangian $L$ in $\Sigma$ and $N \geq 1$, write $C^\infty([0,1]^N \times \bR/\bZ \times \Sigma; L)$ for the space of $H$ in $C^\infty([0,1]^N \times \bR/\bZ \times \Sigma)$ such that $H^\tau \in C^\infty(\bR/\bZ \times \Sigma; L)$ for every $\tau \in [0,1]^N$.

\subsubsection{Nondegeneracy} A periodic orbit $S = \{x_1, \ldots, x_d\}$ is \textbf{nondegenerate} if the Poincar\'e return map $D\phi^d(x_1): T_{x_1}\Sigma \to T_{x_1}\Sigma$ does not have $1$ as an eigenvalue. The map $\phi$ is \textbf{$d$-nondegenerate} if any periodic orbit with period $\leq d$ is nondegenerate and \textbf{nondegenerate} if it is $d$-nondegenerate for each $d \geq 1$. 

\subsection{The mapping torus construction}\label{subsec:mapping_torii} Fix any $\phi \in \Diff(\Sigma, \Omega)$. Then the \textbf{mapping torus}, denoted by $M_\phi$, is the quotient of $[0, 1]_t \times \Sigma$ by the identification $(1, p) \sim (0, \phi(p))$. It carries a vector field $R$ induced by the vector field $\partial_t$ on $[0,1] \times \Sigma$. The periodic orbits of $R$ correspond to periodic orbits of $\phi$. The mapping torus also carries a closed two-form $\Omega_\phi$ induced from $\Omega$ and a closed one-form $dt$. 

Write $A := \int_{\Sigma} \Omega$ for the area of $\Sigma$. A map $\phi \in \Diff(\Sigma, \Omega)$ is \textbf{rational} if and only if $A^{-1} \cdot [\Omega_\phi] \in H^2(M_\phi; \bR)$ is a rational class. Given any two Hamiltonian isotopic maps $\phi$ and $\phi'$, $\phi$ is rational if and only if $\phi'$ is. 

This follows from direct computation which uses the following identification of their mapping torii. Fix a Hamiltonian $H$ such that $\phi' = \phi \circ \psi^1_H$. Define
$$M_H: M_\phi \to M_{\phi'}$$
as the map induced by the diffeomorphism
$$(t, x) \mapsto (t, (\psi^t_H)^{-1}(x))$$
on $[0,1] \times \Sigma$. Pulling back $\Omega_{\phi'}$ by $M_H$, we find 
$M_H^*\Omega_{\phi'} - \Omega_\phi = d(H dt)$, an exact two-form. This implies $\Omega_{\phi'}$ is a real multiple of a rational class if and only $\Omega_\phi$ is. The maps $M_H$ will make frequent appearances below.

Hamiltonian isotopy is an equivalence relation on $\Diff(\Sigma, \Omega)$, and an equivalence class of this relation is called a Hamiltonian isotopy class. A \textbf{rational} Hamiltonian isotopy class is one where every element in the class is rational. The following lemma states a necessary and sufficient homological condition for $\phi \in \Diff(\Sigma, \Omega)$ to be Hamiltonian. 

\begin{lem} \label{lem:hamiltonian_criterion}
Suppose $\Sigma$ has non-empty boundary. Fix a map $\phi \in \Diff(\Sigma, \Omega)$ which is isotopic through area-preserving diffeomorphisms to the identity. Then $\phi$ is Hamiltonian if and only if the two-form $\Omega_\phi$ on the mapping torus $M_\phi$ is exact. 
\end{lem}

\begin{proof}
We start with the simpler direction. Suppose $\phi$ is Hamiltonian and fix a Hamiltonian $H \in C^\infty(\bR/\bZ \times \Sigma)$ such that $\phi = \psi^1_H$. Then the family $\{\psi^t_H\}_{t \in [0,1]}$ gives an isotopy from the identity to $\phi$. As discussed above, this defines a diffeomorphism
$$M_H: \bR/\bZ \times \Sigma \to M_{\phi}$$
pulling back $\Omega_\phi$ to $\Omega + d(H dt)$. This is an exact two-form, so $\Omega_\phi$ is exact. Now we prove the other direction, which is more difficult, but analogous to proofs of the related assertion that isotopies of zero flux are Hamiltonian. Suppose $\Omega_\phi$ is exact. Our goal is to show that $\phi$ is Hamiltonian. We will do this in $3$ steps. 

\noindent\textbf{Step 1:} This step conducts some basic preparations using the assumption that $\Omega_\phi$ is exact. Let $\{\phi^t\}_{t \in [0,1]}$ denote any isotopy in $\Diff(\Sigma, \Omega)$ such that $\phi^0 = \text{id}$ and $\phi^1 = \phi$, and suppose it is constant near $0$ and $1$. Write $X_t := \frac{\partial}{\partial t}\phi^t$ for the time-dependent vector field generating this isotopy, and write $\lambda_t := \Omega(-, X_t)$ for every $t$. The area-preserving condition implies that each $\lambda_t$ is a closed one-form. We define a diffeomorphism
$$F: \bR/\bZ \times \Sigma \to M_\phi$$
to be the map induced by the diffeomorphism
$$(t, x) \mapsto (t, (\phi^t)^{-1}(x))$$
on $[0,1] \times \Sigma$. Then we compute
$$F^*\Omega_\phi = \Omega + \lambda_t \wedge dt.$$

Since $\Omega_\phi$ and $\Omega$ are exact, it follows that $\lambda_t \wedge dt$ is exact. Choose a primitive 
$$\tau := f_t \wedge dt - \alpha_t$$
where $\{f_t\}_{t \in \bR/\bZ}$ is a smooth loop of functions and $\{\alpha_t\}_{t \in [0,1]}$ is a smooth loop of closed one-forms such that $\lambda_t = \frac{d}{dt}\alpha_t + df_t$. By modifying $\phi$ by the Hamiltonian diffeomorphism $\psi^1_f$, we may assume without loss of generality that $f_t \equiv 0$ for each $t$. By replacing $\alpha_t$ with $\alpha_t - \alpha_0$ for each $t$, we may assume without loss of generality that $\alpha_0 = \alpha_1 \equiv 0$. Since $\lambda_t \equiv 0$ near $t = 0$ and $t = 1$, it follows from the latter assumption that $\alpha_t \equiv 0$ near $t = 0$ and $t = 1$. We also observe that since $X_t$ is tangent to the boundary of $\Sigma$, the one-forms $\lambda_t$ vanish on the boundary of $\Sigma$ for each $t$, which in turn implies that $\alpha_t$ vanishes on the boundary of $\Sigma$ for each $t$. 

\noindent\textbf{Step 2:} This step defines a new isotopy from the identity to $\phi$. For each $t \in [0,1]$ define an autonomous vector field $Z_t$ by the equation $\Omega(-, Z_t) = -((\phi^t)^{-1})^*\alpha_t$, and for each $s \in [0,1]$ write $\eta^{s,t}$ for its time-$s$ flow. Note that this flow exists for all time since $\alpha_t$ vanishes on the boundary for each $t$, so $Z_t$ is always tangent to the boundary. For each $t \in [0,1]$, set $\psi^{t} := \eta^{1,t} \circ \phi^t$, and write $Y_t$ for the time-dependent vector field generating this isotopy. Note that $Z_t \equiv 0$ for $t$ near $0$ and $1$, so $\{\psi^{t}\}_{t \in [0,1]}$ is an area-preserving isotopy from the identity to $\phi$, which is constant for $t$ near $0$ and $1$. Use this to define a new diffeomorphism $G: \bR/\bZ \times \Sigma \to M_\phi$ with $G^*\Omega_\phi = \Omega + \widetilde{\lambda}_t \wedge dt$. 

\noindent\textbf{Step 3:} This step shows that the new isotopy from the previous step is Hamiltonian. We claim that the closed one-form $\widetilde{\lambda}_t$ is exact for every $t$. This is equivalent to showing that the one-form $\tau_t := \int_0^t \widetilde{\lambda}_s ds$ is exact for every $t$. The one-form $\tau_t$ is exact if and only if it integrates to $0$ on any closed smooth loop $\gamma: [0, 1] \to \Sigma$. We compute
\begin{align*}
\int_{\gamma} \tau_t &= \int_0^1 \int_0^t \widetilde{\lambda}_s(\dot\gamma(u))) ds du = \int_0^1 \int_0^t \Omega(\dot\gamma(u), Y_s) ds du = \int_{\psi^{[0, t]}(\gamma)} \Omega. 
\end{align*}

The second equality follows from the fact that $\widetilde{\lambda}_t = \Omega(-,Y_t)$ for each $t$. On the right-hand side, $\psi^{[0,t]}(\gamma)$ denotes the cylinder swept out by $\gamma$ during the isotopy $\psi$ from times $0$ to $t$. The third equality uses the fact that the isotopy preserrves $\Omega$. The cylinder $\psi^{[0,t]}(\gamma)$ is homotopic relative to its boundary to the concatenation of the cylinders $\phi^{[0,t]}(\gamma)$ and $\eta^{[0,1], t}(\phi^t(\gamma))$. We use this fact and Stokes' theorem to simplify the right-hand side further: 
\begin{align*}
\int_{\gamma}\tau_t &= \int_{\phi^{[0,t]}(\gamma)} \Omega + \int_{\eta^{[0,1], t}(\phi^t(\gamma))} \Omega \\
&= \int_\gamma (\int_0^t \lambda_s ds) - \int_{\phi^t(\gamma)} (\phi^{-t})^*\alpha_t \\
&= \int_{\gamma} (\int_0^t \lambda_s ds) - \alpha_t = 0. 
\end{align*}

Since the closed one-form $\widetilde{\lambda}_t$ is exact for every $t$ and it is equal to $\Omega(-, Y_t)$, it follows that the time-dependent vector field $Y_t$ is a Hamiltonian vector field for every $t$. Therefore, the original map $\phi$ is Hamiltonian. 
\end{proof}

\section{Area-preserving maps of closed surfaces preserving Lagrangians} \label{sec:maps_preserving_lagrangians}

Fix for the remainder of this section a closed, smooth oriented surface $\Sigma$ with area form $\Omega$ of area $B$ and a Lagrangian $L \subset \Sigma$. The purpose of this section is to prove several technical results moving towards generic density and equidistribution for area-preserving maps of $\Sigma$ preserving $L$. That is, we do not need to show that a generic area-preserving map preserving $L$ has dense/equidistributed periodic orbits for the proofs of our main results, but we need technical results which get almost all the way there. 

\subsection{Results towards generic density} We prove some propositions here which will be used later to prove Theorems~\ref{thm:boundary_density} and \ref{thm:boundary_monotone_density}. 

\begin{prop} \label{prop:closing_lemma}
Fix a rational map $\phi \in \Diff(\Sigma, \Omega; L)$. For any open set $U \subset \Sigma$, there exists an arbitrarily $C^\infty$-small Hamiltonian $H \in C^\infty(\bR/\bZ \times \Sigma; L)$, compactly supported in $\bR/\bZ \times U$, such that the perturbed map $\phi^H = \phi \circ \psi^1_H$ has a periodic point in $U$. 
\end{prop}

\begin{proof}
Choose a non-negative Hamiltonian $H \neq 0$, which is compactly supported in $(0,1) \times U \subset \bR/\bZ \times \Sigma$, and vanishes on $\bR/\bZ \times L$. Repeating verbatim the arguments of \cite[\S$5$]{CGPZ21} implies that for some $\delta \in [0, 1]$, the perturbed map $\phi^{\delta H} = \phi \circ \psi^1_{\delta H}$ has a periodic point in $U$, which proves the proposition. 

We sketch the main points of the argument here for the convenience of the reader\footnote{For the experts: The argument given assumes for simplicity that $H$ is disjoint from the reference cycles used to construct the PFH spectral invariants.}. It uses spectral invariants from periodic Floer homology. There exists a sequence of positive integers $d_m \to \infty$, depending only on the Hamiltonian isotopy class of $\phi$, and for each $\delta \in [0,1]$ a corresponding sequences of real-valued invariants of $\phi^{\delta H}$ called \emph{PFH spectral invariants}. This sequence is denoted by $c_{d_m}(\phi^{\delta H})$; each spectral invariant is equal to an ``action'' of an orbit set of $\phi^{\delta H}$ of period $d_m$. The spectral invariants are quite sensitive to Hamiltonian perturbations. For any Hamiltonian $K$, the invariants for $\phi$ and $\phi^K = \phi \circ \psi^1_K$ satisfy the following asymptotic relation: 
\begin{equation}\label{eq:weyl}\lim_{m \to \infty} \frac{c_{d_m}(\phi^K) - c_{d_m}(\phi)}{d_m} = B^{-1}\int_{M_\phi} K\,dt \wedge \Omega_\phi > 0.\end{equation}

Suppose for the sake of contradiction that $\phi^{\delta H}$ does not have a periodic point in $U$. Formal properties of the PFH spectral invariants then imply that for every $m$, $c_{d_m}(\phi^{\delta H})$ does not depend on $\delta$. The identity \eqref{eq:weyl} implies that the spectral invariants cannot be constant in $\delta$ as claimed, so we arrive at a contradiction and some $\phi^{\delta H}$ must have a periodic point in $U$. 
\end{proof}

\begin{prop} \label{prop:closing_lemma2}
Fix a map $\phi \in \Diff(\Sigma, \Omega; L)$. Assume that $L$ is inessential. For any open set $U \subset \Sigma$, there exists an arbitrarily $C^\infty$-close map $\phi' \in \Diff(\Sigma, \Omega; L)$ such that $\phi'$ has a periodic point in $U$. 
\end{prop}

Proposition~\ref{prop:closing_lemma2} is immediate from Proposition~\ref{prop:closing_lemma} and the following proposition, which shows that area-preserving maps preserving $L$ can be approximated by rational ones, as long as $L$ is inessential. 

\begin{prop} \label{prop:rational_maps_dense}
Assume that $L$ is inessential. Then the set of rational maps preserving $L$ is $C^\infty$-dense in $\Diff(\Sigma, \Omega; L)$. 
\end{prop}

\begin{proof}
The proof closely follows the argument for the full group $\Diff(\Sigma, \Omega)$ from \cite{CGPZ21}. For any given map $\phi \in \Diff(\Sigma, \Omega)$, the cohomology class $[\Omega_\phi] \in H^2(M_\phi; \mathbb{Z})$ has pairing with the image of the class $[\Sigma]$ in $H_2(M_\phi; \mathbb{Z})$ equal to the area $A$ of the surface. It follows that $\phi$ is rational if and only if the integral of $\Omega_\phi$ on any $2$-dimensional integral cycle in $M_\phi$ is a rational multiple of $B$. 

Choose a union of pairwise disjoint embedded disks $\mathbb{D} \subset \Sigma$ bounding the Lagrangian $L$. Write $Z$ for the surface with boundary $L$ constructed by removing the interior of $\mathbb{D}$. Fix any map $\phi \in \Diff(\Sigma, \Omega; L)$. Define $K$ to be the kernel of the map
\begin{equation*} \phi_* - \text{Id}: H_1(\Sigma; \mathbb{Z}) \to H_1(\Sigma; \mathbb{Z}). \end{equation*}

The Mayer--Vietoris sequence and the fact that $L$ bounds a set of disjoint disks implies that the map $H_1(Z; \bZ) \to H_1(\Sigma; \bZ)$ induced by the inclusion $Z \hookrightarrow \Sigma$ is surjective. It follows that there exists a set of oriented closed curves $c_1, \ldots, c_m$ which form a basis of $K$ and are all contained in the interior of $Z$. 

For each $i \in \{1, \ldots, m\}$, there exists a smooth integral $2$-chain $S_i$ in $\Sigma$ such that $\partial S_i = \phi(c_i) - c_i$. Then for each $i$, we define a closed $2$-chain $\overline{S}_i$ in $M_\phi$ by summing the image of $S_i$ in 
$$\{0\} \times \Sigma \subset M_\phi$$
with the image of $[0,1] \times c_i \subset [0,1] \times \Sigma$ after projection to the mapping torus. It follows from the Mayer--Vietoris sequence that $H_2(M_\phi; \bZ)$ is generated by $[\Sigma]$ and $[\overline{S}_1], \ldots, [\overline{S}_m]$. 

Now fix any closed $1$-form $\lambda$ on $\Sigma$ which restricts to $0$ on $L$. This defines a unique vector field $X_\lambda$ on $\Sigma$ defined by the equation $\Omega(X_\lambda, -) = \lambda(-)$. Write $\{\psi^s_\lambda\}_{s \in \bR}$ for the flow of $X_\lambda$, which preserves $\Omega$ and $L$ since $\lambda$ is closed and restricts to $0$ on $L$, and write $\phi' = \phi \circ \psi^1_\lambda$. The arguments of \cite[\S~$5$]{CGPZ21} produce a set of closed integral $2$-chains $\overline{S}_1', \ldots, \overline{S}_m'$ such that $[\Sigma], [\overline{S}_1'], \ldots, [\overline{S}_m']$ span $H_2(M_{\phi'}; \bZ)$ and moreover
$$\int_{\overline{S}_i'} \Omega_{\phi'} = \int_{S_i} \Omega_\phi + \int_{c_i} \lambda$$
for every $i$. From what was said at the beginning of the proof, $\phi'$ is rational if and only if 
$$\int_{\overline{S}_i'} \Omega_{\phi'} \in B \cdot \mathbb{Q}$$
for every $i$. Now the proposition follows if we show that there exist $C^\infty$-small closed $1$-forms $\lambda$ on $\Sigma$ which both restrict to $0$ on $L$ and satisfy 
\begin{equation} \label{eq:rationality} \int_{S_i} \Omega_\phi + \int_{c_i} \lambda \in B \cdot \mathbb{Q} \end{equation}
for every $i$. 

To construct the $1$-forms $\lambda$, we begin by choosing any $C^\infty$-small closed $1$-form $\lambda_0$ on $\Sigma$ which satisfies \eqref{eq:rationality} for every $i$. Choose a contractible open neighborhood $U$ of the union $\mathbb{D}$ of disks bounding $L$ which is disjoint from any of the curves $c_i$. It follows that the restriction of $\lambda_0$ to $U$ is exact, and equal to $df$ for some smooth function $f$ on $U$, which itself can be taken to be $C^\infty$-small when $\lambda_0$ is, for instance by invoking the Poincar\'e lemma. Define a smooth cutoff function $\chi$ which is equal to $1$ on a compact subset $E \subset U$ containing $\mathbb{D}$ and $0$ on the complement of an open neighborhood of $E \subset U$. Then define $\lambda = \lambda_0 - d(\chi f)$. It follows that $\lambda$ is closed, restricts to $0$ on $L$, has the same integrals as $\lambda_0$ on the curves $c_i$, and remains $C^\infty$ small since the $C^\infty$ norm of $\lambda$ depends only on the $C^\infty$ norms of $\lambda_0$, $\chi$, and $f$. This completes the proof of the proposition. 
\end{proof}

\subsection{Results towards generic equidistribution} \label{subsec:lag_equidistribution} We prove the following technical proposition, which asserts that ``nearly equidistributed'' orbit sets can be created by $C^\infty$-small perturbations in $\Diff(\Sigma, \Omega; L)$. It will be used later to prove Theorems~\ref{thm:boundary_equidistribution} and \ref{thm:boundary_monotone_equidistribution}. 

\begin{prop}\label{prop:lag_nearly_equidistributed}
Fix any rational $\phi \in \Diff(\Sigma, \Omega; L)$. Fix any $\epsilon > 0$ and smooth functions $f_1, \ldots, f_N$ on $\Sigma$ which are locally constant on $L$. Then for $C^\infty$-dense $H \in C^\infty(\bR/\bZ \times \Sigma; L)$, the perturbation $\phi^H$ has an orbit set $\cO$ such that for each $i \in \{1, \ldots, N\}$, 
$$|\cO(f_i)/|\cO| - B^{-1}\int_\Sigma f_i\,\Omega| < \epsilon.$$
\end{prop}

The proof of Proposition~\ref{prop:lag_nearly_equidistributed} is identical to the proof of \cite[Proposition~$4.2$]{equidistribution}, replacing the use of the generic nondegeneracy result \cite[Lemma~$3.2$]{equidistribution} with Lemma~\ref{lem:generic_nondegeneracy} below. The proof, just like the proof of Proposition~\ref{prop:closing_lemma}, uses PFH spectral invariants. The rough strategy is to substitute $K = H^\tau$ into the asymptotic formula \eqref{eq:weyl}, where $\phi^{H^\tau}$ of Hamiltonian perturbations parameterized by $\tau \in [0,1]^N$, and then differentiate with respect to $\tau$. We refer the reader to \cite{CGPZ21} or \cite[\S$2$]{equidistribution} for a more detailed discussion of PFH spectral invariants and the Weyl law \eqref{eq:weyl}. The family $H^\tau$ is constructed by taking a $C^\infty$-small perturbation of an $N$-parameter family 
$$F^\tau = \sum_{i=1}^N \tau_i f_i$$
of Hamiltonians explicitly constructed from $f_1, \ldots, f_N$. It is essential that the functions $f_1, \ldots, f_N$ are locally constant on $L$ so that $F^\tau$ is locally constant on $L$ for every $\tau \in [0,1]^N$. We require that the perturbed maps $\phi^{H^\tau}$ are nondegenerate for a full measure set of $\tau \in [0,1]^N$.  The existence of a suitable perturbation $H^\tau \in C^\infty([0,1]^N \times \bR/\bZ \times \Sigma; L)$ requires Lemma~\ref{lem:generic_nondegeneracy} below.

\subsubsection{Generic nondegeneracy} We establish here some transversality results which are necessary to prove Proposition~\ref{prop:lag_nearly_equidistributed} and Theorems~\ref{thm:boundary_equidistribution} and \ref{thm:boundary_monotone_equidistribution}. 

\begin{lem} \label{lem:generic_nondegeneracy}
Fix any $N \geq 1$ and $\phi \in \Diff(\Sigma, \Omega; L)$. Then for $C^\infty$-generic $H \in C^\infty([0,1]^N \times \bR/\bZ \times \Sigma; L)$, there exists a full measure set of $\tau \in [0,1]^N$ such that $\phi^{H^\tau}$ is nondegenerate. 
\end{lem}

\begin{lem}\label{lem:make_nondegenerate}
Fix any $\phi \in \Diff(\Sigma, \Omega; L)$ and any finite set $S_1, \ldots, S_k$ of periodic orbits of $\phi$. Then there exist arbitrarily $C^\infty$-small Hamiltonians $H \in C^\infty(\bR/\bZ \times \Sigma; L)$ such that, for each $i \in \{1, \ldots, k\}$, $S_i$ is a nondegenerate periodic orbit of $\phi^H$. 
\end{lem}

Lemma~\ref{lem:make_nondegenerate} follows from the same constructions used to prove Lemma~\ref{lem:generic_nondegeneracy} and we omit the proof. Lemma~\ref{lem:generic_nondegeneracy} follows from combining the following two lemmas. They respectively address the nondegeneracy of periodic points inside $L$ and outside $L$; the perturbation schemes required in each case are slightly different. 

\begin{lem} \label{lem:generic_nondegeneracy_inside_l}
Fix any $N \geq 1$ and $\phi \in \Diff(\Sigma, \Omega; L)$. Then for $C^\infty$-generic $H \in C^\infty([0,1]^N \times \bR/\bZ \times \Sigma; L)$, there exists a full measure set of $\tau \in [0,1]^N$ such that any periodic orbit of $\phi^{H^\tau}$ which is contained in $L$ is nondegenerate. 
\end{lem}

\begin{lem}\label{lem:generic_nondegeneracy_outside_l}
Fix any $N \geq 1$, and $\phi \in \Diff(\Sigma, \Omega; L)$. Then for $C^\infty$-generic $H \in C^\infty([0,1]^N \times \bR/\bZ \times \Sigma; L)$, there exists a full measure set of $\tau \in [0,1]^N$ such that every periodic orbit of $\phi^{H^\tau}$ which is not contained in $L$ is nondegenerate. 
\end{lem}

\subsubsection{Proof of Lemma~\ref{lem:generic_nondegeneracy_inside_l}} The proof of Lemma~\ref{lem:generic_nondegeneracy_inside_l} proceeds in $4$ steps. 

\noindent\textbf{Step $1$:} For any $l \geq 3$, define $\mathcal{N}^l$ to be the space of tuples $(H, \tau, S)$ where $H \in C^l([0,1]^N \times \bR/\bZ \times \Sigma; L)$, $\tau \in [0,1]^N$, and $S$ is a simple periodic orbit of $\phi^\tau$ which lies in $L$. The space of such simple periodic orbits is topologized as a subset of the disjoint union $\bigsqcup_{d \geq 1} L^d$ of products of $L$, which gives $\mathcal{N}^l$ a natural topology. For any $d \geq 1$, we denote by $\mathcal{N}^{l,d}$ the connected component of $\mathcal{N}^l$ consisting of tuples $(H, \tau, S)$ with $|S| = d$. This step proves the following lemma using the Banach manifold implicit function theorem. 

\begin{lem} \label{lem:generic_nondegeneracy_inside_l2}
For any $l \geq 3$ and $d \geq 1$, the space $\mathcal{N}^{l,d}$ has the structure of a Banach manifold of class $C^{l - 1}$ such that the projection
$$\mathcal{N}^{l,d} \to C^l([0,1]^N \times \bR/\bZ \times \Sigma; L)$$
is a $C^{l-1}$ Fredholm map of index $N$. 
\end{lem}

\begin{proof}[Proof of Lemma~\ref{lem:generic_nondegeneracy_inside_l2}]
Let $\Delta^d \subset L^d$ denote the ``thick diagonal'' consisting of tuples of points $(x_1, \ldots, x_d) \in L^d$ such that $x_i = x_j$ for some $i \neq j$. Define a map
$$\Psi: C^l([0,1]^N \times \bR/\bZ \times \Sigma; L) \times [0,1]^N \times (L^d \setminus \Delta^d) \to (L^d)^2$$
sending a tuple $(H, \tau, S = \{x_1, \ldots, x_d\}$) to the pair of sets
$$((x_1, \ldots, x_d), (\phi^\tau(x_d), \phi^\tau(x_1), \ldots, \phi^\tau(x_{d-1}))).$$

Write $Z \subset (L^d)^2$ for the diagonal. Then by definition, $\mathcal{N}^{l,d} = \Psi^{-1}(Z)$. The linearization $D\Psi$ sends a variation $h$ in the $H$ direction at a point $(H, \tau, S = \{x_1, \ldots, x_d\}) \in \mathcal{N}^{l,d}$ to the pair 
$$((0, \ldots, 0), (V_h(x_d), V_h(x_1), \ldots, V_h(x_{d-1}))) \in \Big(\bigoplus_{i=1}^d T_{x_i}L\Big)^2.$$
Here $V_h$ denotes the derivative $\partial_z|_{z= 0} (\phi \circ \phi^1_{(H_z)^\tau})$ for any family $\{H_z\}_{z \in (-1, 1)}$ in $C^l([0,1]^N \times \bR/\bZ \times \Sigma; L)$ such that $\frac{\partial}{\partial z}|_{z = 0} H_z = h$. It is straightforward to choose $h$ so that $$(V_h(x_d), V_h(x_1), \ldots, V_h(x_{d-1})) \in \bigoplus_{i=1}^d T_{x_i}L$$ is any tuple of tangent vectors. For a given tuple $(H, \tau, S)$, we can choose a family $\{H_z\}$ which varies the time-one map of the Hamiltonian $H$ in the neighborhood of a point $x_i \in S$ by a small hyperbolic map fixing $x_i$, and leaving $L$ invariant. This produces $h$ so that $V_h(x_i) \neq 0$ and $V_h(x_j) = 0$ for all $j \neq i$. A similar construction is written out in more detail in the proof of Lemma~\ref{lem:generic_nondegeneracy_inside_l3} below.  

It follows that the map $\Psi$ is transverse to the diagonal $Z$. The implicit function theorem for Banach manifolds implies that $\mathcal{N}^{l,d}$ is a Banach submanifold of class $C^{l-1}$ and codimension $d$ in $C^l([0,1]^N \times \bR/\bZ \times \Sigma; L) \times [0,1]^N \times (L^d \setminus \Delta^d)$. The projection
$$C^l([0,1]^N \times \bR/\bZ \times \Sigma; L) \times [0,1]^N \times (L^d \setminus \Delta^d) \to C^l([0,1]^N \times \bR/\bZ \times \Sigma; L)$$
is Fredholm and $C^l$ of index $N + d$, from which we conclude that the projection
$$\mathcal{N}^{l,d} \to C^l([0,1]^N \times \bR/\bZ \times \Sigma; L)$$
is Fredholm and $C^{l-1}$ of index $N$. 

\end{proof}

\noindent\textbf{Step $2$:} Define $\mathcal{N}_{\text{bad}}^l$ to be the set of points $(H, \tau, S) \in \mathcal{N}^l$ for which $S$ or one of its iterates is not nondegenerate. We denote by $\mathcal{N}^{l,d}_{\text{bad}}$ the connected component of $\mathcal{N}^{l}_{\text{bad}}$ consisting of tuples $(H, \tau, S)$ with $|S| = d$. This step proves the following lemma. 

\begin{lem}\label{lem:generic_nondegeneracy_inside_l3}
For any $l \geq 3$ and $d \geq 1$, the space $\mathcal{N}^{l,d}_{\text{bad}}$ is a countable union of $C^{l - 2}$ Banach submanifolds of $\mathcal{N}^{l,d}$ of codimension at least $1$. 
\end{lem}

\begin{proof}[Proof of Lemma~\ref{lem:generic_nondegeneracy_inside_l3}]

We define a universal fiber bundle $\mathcal{E} \to \mathcal{N}^{l,d}$ of class $C^{l-1}$ by setting its fiber at $(H, \tau, S = \{x_1, \ldots, x_d\})$ to equal the group of linear automorphisms of $T_{x_1}L$, which is naturally isomorphic to the multiplicative group $\mathbb{R}^*$ of nonzero real numbers. There exists a natural $C^{l-1}$ section $\alpha: \mathcal{N}^{l,d} \to \mathcal{E}$ taking $(H, \tau, S)$ to the restriction of the return map to $T_{x_1} L$.  

It is convenient to trivialize $\mathcal{E}$. Define a $C^{l-1}$-smooth map
$$\Psi: \mathcal{E} \to \bR^*$$
sending a tuple $(H, \tau, S = \{x_1, \ldots, x_d\}, \rho)$, where $\rho$ is a linear automorphism of $T_{x_1}L$, to its determinant $\lambda \in \mathbb{R}_*$. For a given tuple $(H, \tau, S)$, $S$ and all its iterates are nondegenerate if and only if $(\Phi \circ \alpha)(H, \tau, S) \not\in \{-1, 1\}$. This is because $L$ is $\phi^{H^\tau}$-invariant, so the tangent space $T_{x_1} L$ is invariant under the return map of $S$, which implies that an iterate of $S$ is degenerate if and only if the return map acts by multiplication by $\pm 1$ on $T_{x_1} L$. The lemma then follows from showing that $1$ and $-1$ are regular values of $\Psi \circ \alpha$ and then applying the implicit function theorem. 

The fact that $1$ and $-1$ are regular values of $\Psi \circ \alpha$ follows from the fact that the linearization $D(\Psi \circ \alpha)$ does not vanish at any point, which we now prove. Fix any $(H, \tau, S = \{x_1, \ldots, x_d\}) \in \mathcal{N}^{l,d}$. Write $\gamma$ for the component of $L$ containing $x_1$. The Weinstein tubular neighborhood theorem shows that for some small $\delta > 0$, there is a smooth embedding 
$$F: (\bR/\bZ)_t \times (-\delta, \delta)_s \to \Sigma$$
satisfying the following properties:
\begin{itemize}
\item $F(0, 0) = x_1$ and the map $t \mapsto F(t, 0)$ parameterizes $\gamma$. 
\item $F$ is a diffeomorphism onto an open neighborhood of $\gamma$ which is disjoint from any other component of $L$. 
\item $F^*\Omega = ds \wedge dt.$
\end{itemize}

Since $\phi^{\tau}$ preserves $L$, there is some positive $\delta' \ll \delta$ such that the conjugated map
$$F^{-1} \circ (\phi^{\tau})^d \circ F: \bR/\bZ \times (-\delta', \delta') \to \bR/\bZ \times (-\delta, \delta)$$
is a well-defined smooth embedding preserving the area form $ds \wedge dt$. Define an autonomous Hamiltonian 
$$G: \bR/\bZ \times (-\delta', \delta') \to \bR,$$
compactly supported in a small neighborhood of $(0, 0)$, which satisfies the formula $G(t, s) = -st$ sufficiently close to $(0, 0)$. Note that the function $G \circ F^{-1}$ extends by zero to a well-defined autonomous Hamiltonian on $\Sigma$ which vanishes on $L$. For any $z \in (-1, 1)$, write $\psi_z$ for the time one map of $2zG$. If the support of $G$ is sufficiently close to $(0, 0)$, we conclude that for any $z \in (-1, 1)$, the Hamiltonian diffeomorphism $\phi^\tau \circ \psi_z$ has $S$ as a simple periodic orbit and the return map at $x_1 \in L$ has determinant $\lambda \cdot \text{exp}(z)$ on $T_{x_1}L$, where $\lambda = (\Psi \circ \alpha)(H, \tau, S)$. We fix a family of Hamiltonians $\{H_z\}_{z \in (-1, 1)}$ in $C^l([0,1]^N \times \bR/\bZ \times \Sigma; L)$ so that the time one map of $H_z^\tau$ equals $\phi^\tau \circ \psi_z$ for every $z \in (-1, 1)$. Then 
$$\frac{\partial}{\partial z}|_{z = 0} (\Psi \circ \alpha)(H_z, \tau, S) = \lambda \neq 0$$
so the linearization $D(\Psi \circ \alpha)$ does not vanish at any point in the domain as desired. This concludes the proof of the lemma. 
\end{proof}

\noindent\textbf{Step $3$:} This step proves the following lemma.

\begin{lem}\label{lem:generic_nondegeneracy_inside_l1}
For any $l \geq 3$, there exists a generic subset of $H \in C^l([0,1]^N \times \bR/\bZ \times \Sigma; L)$ such that 
$$\text{measure}(\{\tau \in [0,1]^N\,|\,\text{Any periodic orbit of $\phi^\tau$ in $L$ is nondegenerate}\}) = 1.$$
\end{lem}

\begin{proof}[Proof of Lemma~\ref{lem:generic_nondegeneracy_inside_l1}]
Introduce the notation 
$$\Pi: \mathcal{N}^{l,d} \to C^l([0,1]^N \times \bR/\bZ \times \Sigma; L)$$
for the projection. Lemmas~\ref{lem:generic_nondegeneracy_inside_l2} and \ref{lem:generic_nondegeneracy_inside_l3}, along with the Sard--Smale theorem implies that there is a generic set $E_d \subset C^l([0,1]^N \times \bR/\bZ \times \Sigma; L)$ satisfying the following properties:
\begin{itemize}
    \item For any $H \in E_d$, the preimage $\Pi^{-1}(H)$ is a manifold of class $C^{l-1}$ with dimension $N$. 
    \item For any $H$ in $E_d$, the preimage $\Pi^{-1}(H) \cap \mathcal{N}^{l,d}_{\text{bad}}$ is a countable union of submanifolds of class $C^{l-2}$ of $\Pi^{-1}(H)$ with codimension at least $1$. 
\end{itemize}

Fix any $H \in E_d$. Sard's theorem implies that there is a full measure set of $\tau \in [0,1]$ which are simultaneously regular values of the projections $\Pi^{-1}(H) \to [0,1]^N$ and $\Pi^{-1}(H) \cap \mathcal{N}^{l,d}_{\text{bad}} \to [0,1]^N$. For each regular value $\tau$, the preimage of $\tau$ under this projection does not intersect $\Pi^{-1}(H) \cap \mathcal{N}^{l,d}_{\text{bad}}$ since it is a union of submanifolds of dimension $\leq N - 1$. It follows by definition that, for any $H \in E_d$, any orbit of $\phi^\tau$ lying in $L$ is nondegenerate for a full measure set of $\tau \in [0,1]^N$. Taking the intersection of the sets $E_d$ across all $d$ yields the generic set desired by Lemma~\ref{lem:generic_nondegeneracy_inside_l1}. 
\end{proof}

\noindent\textbf{Step $4$:} This step concludes the proof by combining Lemma~\ref{lem:generic_nondegeneracy_inside_l1} with a quick formal argument. It is identical to the argument proving \cite[Lemma~$3.2$]{equidistribution} from the counterpart of Lemma~\ref{lem:generic_nondegeneracy_inside_l1} in that work. For any integer $d \geq 1$ and $\delta > 0$, write $E(d, \delta)$ for the set of Hamiltonians in $C^\infty([0,1]^N \times \bR/\bZ \times \Sigma; L)$ such that, for any $H \in E(d, \delta)$, 
\begin{equation} \label{eq:parametric1} \text{measure}(\{\tau \in [0,1]^N\,|\,\text{Any orbit of $\phi^\tau$ in $L$ of period $\leq d$ is nondegenerate}\}) > 1 - \delta. \end{equation}

We claim that $E(d, \delta)$ is open and dense in $C^\infty([0,1]^N \times \bR/\bZ \times \Sigma; L)$. It is clearly open. For any $l \geq 3$, the fact that $C^l$ functions may be approximated in the $C^l$ topology by smooth functions and Lemma~\ref{lem:generic_nondegeneracy_inside_l1} imply that $E(d,\delta)$ is dense in $C^l([0,1]^N \times \bR/\bZ \times \Sigma; L)$. It follows that it is dense in $C^\infty([0,1]^N \times \bR/\bZ \times \Sigma; L)$ as desired. Now fix a sequence $d_N \to \infty$ and a sequence $\delta_N \to 0$. Then the set
$$E := \bigcap_{N \geq 1} E(d_N, \delta_N)$$
is the generic set desired by Lemma~\ref{lem:generic_nondegeneracy_inside_l}. 

\subsubsection{Proof of Lemma~\ref{lem:generic_nondegeneracy_outside_l}} Fix $N \geq 1$ and $\phi \in \Diff(\Sigma, \Omega; L)$ as in the statement of the lemma. For any $l \geq 0$, write $C^l([0,1]^N \times \bR/\bZ \times \Sigma; L)$ for the Banach space of functions of class $C^l$ on $[0,1]^N \times \bR/\bZ \times \Sigma$ which are locally constant on $L$. The proof of Lemma~\ref{lem:generic_nondegeneracy_outside_l} is a minor modification of the proof of \cite[Lemma~$3.2$]{equidistribution}. It is a consequence of the following three lemmas. 

\begin{lem}\label{lem:generic_nondegeneracy_outside_l1}
For any $l \geq 3$, there is a generic set of $H \in C^l([0,1]^N \times \bR/\bZ \times \Sigma; L)$ such that 
$$\text{measure}(\{\tau \in [0,1]^N\,|\,\text{Any periodic orbit of $\phi^\tau$ in $\Sigma\setminus L$ are nondegenerate}\}) = 1.$$
\end{lem}

For any $l \geq 3$, define $\mathcal{M}^l$ to be the space of tuples $(H, \tau, S)$ where $H \in C^l([0,1]^N \times \bR/\bZ \times \Sigma; L)$, $\tau \in [0,1]^N$, and $S$ is a simple periodic orbit of $\phi^\tau$ which does not lie in $L$. The space of such simple periodic orbits is topologized as a subset of the disjoint union $\bigsqcup_{d \geq 1} (\Sigma \setminus L)^d$ of products of $\Sigma \setminus L$, which gives $\mathcal{M}^l$ a natural topology. For any $d \geq 1$, we denote by $\mathcal{M}^{l,d}$ the connected component of $\mathcal{M}^l$ consisting of tuples $(H, \tau, S)$ with $|S| = d$. 

\begin{lem} \label{lem:generic_nondegeneracy_outside_l2}
For any $l \geq 3$ and $d \geq 1$, the space $\mathcal{M}^{l,d}$ has the structure of a Banach manifold of class $C^{l - 1}$ such that the projection
$$\mathcal{M}^{l,d} \to C^l([0,1]^N \times \bR/\bZ \times \Sigma; L)$$
is a $C^{l-1}$ Fredholm map of index $N$. 
\end{lem}

Define $\mathcal{M}_{\text{bad}}^l$ to be the set of points $(H, \tau, S) \in \mathcal{M}^l$ for which the return map of $\phi^{H^\tau}$ at $S$ has a root of unity as an eigenvalue. For any $d \geq 1$, we denote by $\mathcal{M}^{l,d}_{\text{bad}}$ the connected component of $\mathcal{M}^{l}_{\text{bad}}$ consisting of tuples $(H, \tau, S)$ with $|S| = d$. 

\begin{lem}\label{lem:generic_nondegeneracy_outside_l3}
For any $l \geq 3$ and $d \geq 1$, the space $\mathcal{M}^{l,d}_{\text{bad}}$ is a countable union of $C^{l - 2}$ Banach submanifolds of $\mathcal{M}^{l,d}$ of codimension at least $1$. 
\end{lem}

Lemmas~\ref{lem:generic_nondegeneracy_outside_l1}, \ref{lem:generic_nondegeneracy_outside_l2}, and \ref{lem:generic_nondegeneracy_outside_l3} are counterparts of Lemmas~$5.2$, $5.3$ and $5.4$ in \cite{equidistribution}, respectively, and they are proved by repeating the proofs of these lemmas. The only new difficulty in our setting is that we can only vary the map by Hamiltonian perturbations which are locally constant on $L$, which a priori may not be a large enough set of perturbations to guarantee transversality. However, locally at any point which does not intersect $L$, an arbitrary Hamiltonian perturbation can be replaced by a Hamiltonian perturbation which is locally constant on $L$. As a result, we find that the space of Hamiltonian perturbations which are locally constant in $L$ is large enough to guarantee nondegeneracy for periodic orbits which do not lie in $L$. Lemma~\ref{lem:generic_nondegeneracy_outside_l} now follows from Lemma~\ref{lem:generic_nondegeneracy_outside_l1} by using an analogous argument to the one deducing Lemma~\ref{lem:generic_nondegeneracy_inside_l} from Lemma~\ref{lem:generic_nondegeneracy_inside_l1}.

\section{Area-preserving maps of compact surfaces with boundary}

Fix for the remainder of this section a compact, smooth surface $Z$ with area form $\omega$ and boundary $L = \partial Z$. Write $A := \int_Z \omega$ for the area of $Z$.

\subsection{Capping off} \label{subsec:capping_off}

Fix any $B > A$. By attaching disks to $L$, we construct a smooth surface $\Sigma$ such that $L$ is an inessential Lagrangian inside of $\Sigma$, and an area form $\Omega$ extending $\omega$ with $\int_{\Sigma} \Omega = B$. Write $\gamma_1, \ldots, \gamma_n$ for the connected components of $L$. By the Weinstein tubular neighborhood theorem, there exists $\delta > 0$ and for each $1 \leq i \leq n$ an embedding
$$\tau_i: [0, \delta)_s \times (\bR/\bZ)_t \hookrightarrow Z$$
such that $\tau_i(0, -)$ parameterizes $\gamma_i$ and $\tau_i^*\omega = ds \wedge dt$. Set 
$$r_0 := \sqrt{\frac{B - A}{n\pi}}, \quad r_1 := \sqrt{\frac{B - A + n\delta}{n\pi}}$$
so that $\pi r_0^2 = (B - A)/n$, $\pi r_1^2 = (B - A)/n + \delta$. Let $\mathbb{A} \subset \mathbb{C}$ denote the half-closed annulus defined in polar coordinates $(r, \theta)$ by $r \in [r_0, r_1)$, and let $\mathbb{D} \subset \mathbb{C}$ denote the open disk of radius $r_1$. The map
$$\psi: (s, t) \mapsto (\sqrt{\frac{\pi r_0^2 + s}{\pi}}, 2\pi t)$$
is a symplectomorphism $([0, \delta) \times \bR/\bZ, ds \wedge dt) \to (\mathbb{A}, r dr \wedge d\theta)$. For each $i$, glue the disk $\mathbb{D}$ onto $Z$ via the embedding $\tau_i \circ \psi^{-1}: \mathbb{A} \hookrightarrow Z$. This produces a smooth surface $\Sigma$, and the area form $\omega$ extends to an area form $\Omega$ of area $B$, which is equal to $r dr \wedge d\theta$ on any copy of $\mathbb{D}$.

\subsection{An extension result} \label{subsec:extension}

The following proposition, once established, will allow us to use the results of Section~\ref{sec:maps_preserving_lagrangians} to conclude our main theorems. 

\begin{prop}\label{prop:extension}
For any map $\phi_0 \in \Diff(Z, \omega)$, there exists a map $\phi \in \Diff(\Sigma, \Omega; L)$ which restricts to $\phi_0$ on $Z$. 
\end{prop}

The proof of Proposition~\ref{prop:extension} relies on a series of lemmas. 

\begin{lem}\label{lem:compisotopy}
Fix $\phi_0\in \Diff(Z,\omega)$ and assume that each connected component of $\partial Z$ is $\phi_0$-invariant. Then $\phi_0$ is isotopic through area-preserving diffeomorphisms to some $\psi_0 \in \Diff(Z,\omega)$ which coincides with the identity map near $\partial Z$. 
\end{lem}

\begin{proof}
By the Dehn--Lickorish theorem, there exists a smooth path $\{\phi_s\}_{s\in [0,1]}$ of diffeomorphisms such that $\phi = \phi_0$ and $\psi_0 := \phi_1$ is a composition of compactly supported, area-preserving Dehn twists. For convenience, we assume this path is constant near $s = 0$ and $s = 1$. We use a Moser argument to deform the family $\phi_s$ into a family in $\Diff(Z, \omega)$ while keeping the endpoints fixed.

For any $s \in [0,1]$, write $\omega_s := \phi_s^*\omega$. For any $s \in [0,1]$ and $t \in [0,1]$, write $\omega_{s,t} := (1-t)\omega_s + t\omega$. Note that each of these are a smooth area form on $Z$. Because $Z$ has non-empty boundary, there exists some one-form $\beta$ such that $\omega = d\beta$. Choose a one-parameter family $\{\beta_s\}_{s \in [0,1]}$ of smooth one-forms such that $d\beta_s = \omega_s$ for each $s \in [0,1]$ and $\beta_s \equiv \beta$ near $s = 0$ and $s = 1$. Write $\beta_{s,t} := (1 - t)\beta_s + t\beta$ for each $s \in [0,1]$ and $t \in [0,1]$. We now claim that there exists a smooth two-parameter family of diffeomorphisms $\{\eta_{s,t}\}_{s,t\in [0,1]}$ such that i) $\eta_{s,t} = \operatorname{Id}_Z$ near $s = 0$ and $s = 1$, ii) $\eta_{s,0} = \operatorname{Id}_Z$ for any $s \in [0,1]$, and 
\begin{equation}\label{eq:moser}\eta_{s,t}^*\omega_{s,t} = \omega \end{equation} for any $s$ and $t$. This claim, if true, implies the lemma, since the family $\{\phi_{s,1} := \phi_s \circ \eta_{s,1}\}_{s \in [0,1]}$ is an isotopy in $\Diff(Z, \omega)$ from $\phi_0$ to $\psi_0$. For each fixed $s \in [0,1]$, let $\{X_{s,t}\}_{t \in [0,1]}$ be the time-dependent vector field generating the isotopy $\{\eta_{s,t}\}_{t \in [0,1]}$. For any fixed $s \in [0,1]$, differentiate \eqref{eq:moser} with respect to $t$ to find
\begin{equation*} \mathcal{L}_{X_{s,t}}\omega_{s,t} + \omega - \omega_s = 0. \end{equation*}

Expand the left-hand side using Cartan's formula to deduce
\begin{equation} \label{eq:moser1} d(\omega_{s,t}(X_{s,t}, -) + \beta - \beta_s) = 0. \end{equation}

Since each $\omega_{s,t}$ is an area form, for any fixed $s$ and $t$ there exists a unique vector field $X_{s,t}$ such that $\omega_{s,t}(X_{s, t}, -) = \beta_s-\beta$. If we choose $X_{s,t}$ to solve this equation, then the two-parameter family $\{\eta_{s,t}\}_{s,t\in[0,1]}$ has the following properties. Since $\omega_{s,t}$ varies smoothly in $s$ and $t$ and $\beta_s$ varies smoothly in $s$, the vector fields $\{X_{s,t}\}$ vary smoothly in $s$ and $t$, and therefore $\{\eta_{s,t}\}_{s, t \in [0,1]}$ varies smoothly in $s$ and $t$. Since $\omega_{s,t} \equiv \omega$ and $\beta_s \equiv \beta$ near $s = 0$ and $s = 1$, it follows that $X_{s,t} \equiv 0$ near $s = 0$ and $s = 1$, and therefore $\eta_{s,t} = \operatorname{Id}_Z$ near $s = 0$ and $s = 1$. It is also immediate from the definition that $\eta_{s,0} = \operatorname{Id}_Z$ for each $s \in [0,1]$. Finally, the identity \eqref{eq:moser1} implies \eqref{eq:moser}. This proves the lemma. 

\end{proof}

\begin{lem}\label{lem:zeroflux}
Fix any $\psi\in \Diff(Z,\omega)$ which is isotopic to the identity through area-preserving diffeomorphisms. Then there exists $\psi_1\in\Diff(Z,\omega)$, which coincides with the identity near $\partial Z$, such that the map $\psi_1 \circ\psi$ is a Hamiltonian diffeomorphism.
\end{lem}

\begin{proof}
Fix an isotopy $\{\tau_t\}_{t \in [0,1]}$, constant near $t = 0$ and $t = 1$, such that $\tau_0 = \operatorname{Id}_Z$ and $\tau_1 = \psi$. Write $X_t$ for the time-dependent vector field generating this isotopy, and for each $t$ set $\lambda_t := \omega(-, X_t)$. These are closed one-forms which restrict to $0$ on the boundary, and vanish identically if $t$ is near $0$ or $1$. Fix any small neighborhood $U$ of $\partial Z$ which retracts onto $\partial Z$. Then there exists a smooth family $\{f_t\}_{t \in [0,1]}$ of smooth functions on $U$, locally constant on $\partial Z$ and equal to $0$ for $t$ near $0$ and $1$, such that $\lambda_t|_U = df_t$ for every $t$. Choose a cutoff function $\chi: Z \to [0,1]$ which is equal to $1$ on a neighborhood of $\partial Z$ and compactly supported in $U$. Define a time-dependent Hamiltonian $H_t := \chi f_t$ and let $\tau$ be its time-one map. Then $\tau$ agrees with $\psi$ in a smaller neighborhood of $\partial Z$. The map $\psi_1 := \tau \circ \psi^{-1}$ satisfies the conditions of the lemma. 
\end{proof}

\begin{lem}\label{lem:no_fix_boundary}
Write the inessential Lagrangian $L \subset \Sigma$ as a disjoint union $\cup_{i=1}^k \gamma_i$ of embedded curves. Let $\sigma: \{1, \ldots, k\} \to \{1, \ldots, k\}$ be any bijection from the set $\{1, \ldots, k\}$ to itself. Then there exists $g \in \Diff(\Sigma, \Omega; L)$ which permutes the components of $L$ as prescribed by $\sigma$. That is, for any $i \in \{1, \ldots, k\}$ we have $g(\gamma_i) = \gamma_{\sigma(i)}$. 
\end{lem}

\begin{proof}
For each $i$, write $D_i$ for the embedded closed disk bounded by $\gamma_i$. Note by construction each of the $D_i$ have the same area. It is a straightforward consequence of the isotopy extension theorem that there exists a diffeomorphism $g_0: \Sigma \to \Sigma$ whose restriction to $D_i$ is a diffeomorphism onto $D_{\sigma(i)}$. We now use a Moser-type argument to modify $g_0$ to an area-preserving map which satisfies the same property. Write $\Omega_0 := g_0^*\Omega$. The two-form $\Omega - \Omega_0$ integrates to $0$ over $\Sigma$, so it is exact. Choose any primitive $\sigma_0$. Note also that
$$\int_{D_i} \Omega_0 = \int_{D_{\sigma(i)}} \Omega = \int_{D_i} \Omega$$
for every $i$, so Stokes' theorem implies that $\int_{\gamma_i}\sigma_0 = 0$ for every $i$. This shows that the restriction of $\sigma_0$ to each of the loops $\gamma_i$ is exact. It follows that there exists a smooth function $f$, compactly supported in a neighborhood of $L$, such that $df|_{L} \equiv \sigma_0|_{L}$. Write $\sigma := \sigma_0 - df$. The one-form $\sigma$ is a primitive of $\Omega - \Omega_0$ which restricts to $0$ on $L$.  

The rest of the proof follows well-known arguments. For each $t \in [0,1]$, write $\Omega_t := (1 - t)\Omega_0 + t\Omega$. Define a time-dependent vector field $\{V_t\}_{t \in [0,1]}$ implicitly by the Moser equation $\Omega_t(V_t, -) = -\sigma$. Then the isotopy $\{\tau_t\}_{t \in [0,1]}$ it generates satisfies $\tau_t^*\Omega_t = \Omega_0$ for every $t$. Since $\sigma$ restricts to $0$ on $L$, it follows that $V_t$ is tangent to $L$, so the isotopy fixes each component of $L$. The map $g := g_0 \circ (\tau_1)^{-1}$ is the area-preserving map required by the lemma. 
\end{proof}

\begin{proof}[Proof of Proposition~\ref{prop:extension}]

We prove the proposition in $2$ steps.

\noindent\textbf{Step $1$:} This step constructs the extension under the assumption that extensions exist for maps which preserve each boundary component of $Z$. By Lemma~\ref{lem:no_fix_boundary}, there exists $g \in \Diff(\Sigma, \Omega; L)$ which permutes the components of $L$ in the same way that $\phi_0$ permutes them. Write $g_0$ for its restriction to $Z$. Then $\psi_0 := \phi_0 \circ g_0^{-1} \in \Diff(Z, \omega)$ leaves each boundary component of $Z$ invariant. Fix an extension $\psi \in \Diff(\Sigma, \Omega; L)$ of $\psi_0$. Then $\phi := \psi \circ g \in \Diff(\Sigma, \Omega; L)$ extends $\phi_0$. 

\noindent\textbf{Step $2$:} This step constructs the extension under the assumption that each boundary component of $Z$ is preserved by $\phi_0$. By Lemma~\ref{lem:compisotopy}, there exists some compactly supported $\psi_{0}\in \Diff(Z,\omega_{Z})$ isotopic to $\phi_0$ through area-preserving diffeomorphisms. Applying Lemma \ref{lem:zeroflux} to the map $\psi_0^{-1} \circ \phi_0$, which is isotopic to the identity, there is a compactly supported map $\psi_1 \in \Diff(Z,\omega_{Z})$ such that $\psi_1 \circ \psi_0^{-1} \circ \phi_0$ is Hamiltonian. Fix a smooth Hamiltonian function $H: \bR/\bZ \times Z \to \bR$ such that the time-one map $\psi^1_H$ of the Hamiltonian flow is equal to $\psi_1 \circ \psi_0^{-1} \circ \phi_0$. 

Fix a smooth extension $\widetilde{H}: \bR/\bZ \times \Sigma \to \bR$ of $H$ to the closed surface $\Sigma$. Write $\phi_1 \in \Diff(\Sigma, \Omega)$ for its time-one map. Note that since $\psi_0$ and $\psi_1$ are both compactly supported in the interior of $Z$, the restriction of the map $\phi_1$ to $Z$ coincides with $\phi_0$ in a neighborhood of $\partial Z$. It follows that the map $\phi \in \Diff(\Sigma, \Omega)$ specified by 
\begin{equation*}
\phi(x):=\begin{cases*}
                   \phi_{0}(x) & if  $x\in Z$  \\
                   \phi_{1}(x) & if $x\notin  Z.$
                 \end{cases*}
 \end{equation*}
is well-defined and smooth. This map is the desired extension of $\phi_0$. 
\end{proof}

We conclude with a final lemma determining when extensions of rational maps are rational.

\begin{lem} \label{lem:rationality}
Fix $\phi \in \Diff(\Sigma, \Omega; L)$ and write $\phi_0 \in \Diff(Z, \omega)$ for its restriction to $Z$. Assume that $B$ is a rational multiple of $A$. Then $\phi$ is rational if and only if $\phi_0$ is. 
\end{lem}

\begin{proof}
Fix any collection of $1$-cycles $c_1, \ldots, c_k$ in $Z$ which are disjoint from $L$ and generate $\ker(\phi_* - \text{id}) \subset H_1(\Sigma; \mathbb{Z})$. Fix a collection of $1$-cycles $\gamma_1, \ldots, \gamma_{\ell} \subset L$ which generate $\ker( (\phi|_L)_* - \text{id}) \subset H_1(L; \mathbb{Z})$. For each $i \in \{1, \ldots, k\}$, define a $2$-cycle $S_i \subset M_{\phi}$ by capping off the image of $[0,1] \times c_i \subset [0,1] \times \Sigma$ in $M_{\phi}$ with a $2$-chain in $\{0\} \times \Sigma$ with boundary $\phi(c_i) - c_i$. For each $j \in \{1, \ldots, \ell\}$, the image of $[0,1] \times \gamma_j \subset [0,1] \times \Sigma$ in $M_\phi$ is a $2$-cycle $T_j$. 

The cycles $S_i$ are contained in $M_{\phi_0}$, while the cycles $T_j$ are contained in $\partial M_{\phi_0}$. It follows from the Mayer--Vietoris sequence that $H_2(M_{\phi}; \mathbb{Z})$ is generated by $[\Sigma], [S_1], \ldots, [S_k]$, while $H_2(M_{\phi_0}; \mathbb{Z})$ is generated by $[S_1], \ldots, [S_k], [T_1], \ldots, [T_{\ell}]$. 

The two-form $\Omega_{\phi_0}$ on $M_{\phi_0}$ is the restriction of $\Omega_\phi$ to $M_{\phi_0}$. Also, it restricts to $0$ on $\partial M_{\phi_0}$, so it integrates to $0$ on each cycle $T_j$. It follows that the map $\phi_0$ is rational if and only if 
$$A^{-1}\int_{S_i} \Omega_\phi \in \mathbb{Q}$$
for all $i$. Since $B$ is a rational multiple of $A$, this is true if and only if $B^{-1}\int_{S_i} \Omega_\phi \in \mathbb{Q}$ for all $i$. The integral of $\Omega_\phi$ over $\Sigma$ is equal to $B$, so this is true if and only if $\phi$ is rational. Following this chain of implications yields the desired outcome, that $\phi$ is rational if and only if $\phi_0$ is rational. 
\end{proof}

We assume from now on that $B$ is a rational multiple of $A$, so that Lemma \ref{lem:rationality} is relevant.

\subsection{Proof of generic density}\label{subsec:density_proof}

We prove Theorems~\ref{thm:boundary_density} and \ref{thm:boundary_monotone_density} here. They follow from Propositions~\ref{prop:boundary_closing_monotone} and \ref{prop:boundary_closing} below by standard Baire category arguments. These propositions respectively follow from Propositions~\ref{prop:closing_lemma} and \ref{prop:closing_lemma2} above. 

\begin{prop}\label{prop:boundary_closing_monotone}
Fix a rational map $\phi_0 \in \Diff(Z, \omega)$. For any open set $U \subset Z\, \setminus\,\partial Z$, there exists an arbitrarily $C^\infty$-small Hamiltonian $H \in C^\infty(\bR/\bZ \times Z)$, compactly supported in $\bR/\bZ \times U$, such that the perturbed map $\phi_0^H = \phi_0 \circ \psi^1_H$ has a periodic point in $U$. 
\end{prop}

\begin{proof}
Extend $\phi_0$ to some $\phi \in \Diff(\Sigma, \Omega; L)$ using Proposition~\ref{prop:extension}. By Lemma \ref{lem:rationality}, $\phi$ is rational. Proposition~\ref{prop:closing_lemma} shows that there is an arbitrarily $C^\infty$-small Hamiltonian $H \in C^\infty(\bR/\bZ \times Z$, compactly supported in $\bR/\bZ \times U$, such that the perturbed map $\phi^H = \phi \circ \psi^1_H$ has a periodic point in $U$. The restriction $\phi_0^H = \phi_0 \circ \psi^1_H$ therefore has a periodic point in $U$ and the proposition follows. 
\end{proof}

\begin{prop}\label{prop:boundary_closing}
Fix a map $\phi_0 \in \Diff(Z, \omega)$. For any open set $U \subset Z \setminus \partial Z$, there exists an arbitrarily $C^\infty$-close map $\psi_0 \in \Diff(Z, \omega)$ which has a periodic point in $U$. 
\end{prop}

\begin{proof}
Extend $\phi_0$ to some $\phi \in \Diff(\Sigma, \Omega; L)$ using Proposition~\ref{prop:extension}. Proposition~\ref{prop:closing_lemma2} provides a $C^\infty$-close map $\psi \in \Diff(\Sigma, \Omega; L)$ with a periodic point through $U$. The restriction $\psi_0 \in \Diff(Z, \omega)$ of $\psi$ to $Z$ is the desired $C^\infty$-small perturbation of $\phi_0$. 
\end{proof}

\subsection{Proof of generic equidistribution} \label{subsec:equidistribution_proof}

We prove Theorems~\ref{thm:boundary_equidistribution} and \ref{thm:boundary_monotone_equidistribution} using Proposition~\ref{prop:lag_nearly_equidistributed} from the previous section. As preparation, we introduce the following notation. Fix any $\phi \in \Diff(\Sigma, \Omega; L)$. 

For every orbit set $$\cO = \sum_{k=1}^N a_k \cdot S_k \in \cP_{\bR}(\phi),$$ we define 
$$\cO^{Z} := \sum_{k\text{ such that }S_k \subset Z} a_k \cdot S_k$$
to be the orbit set constructed from $\cO$ by taking only those orbits $S_k$ which lie in $Z \subset \Sigma$. Writing $\phi_0 \in \Diff(Z, \omega)$ for the restriction of $\phi$ to $Z$, we have $\cO^Z \in \cP_{\bR}(\phi_0)$, that is it is an orbit set for $\phi_0$. We begin with a proof of Theorem~\ref{thm:boundary_monotone_equidistribution}. 

\begin{proof}[Proof of Theorem~\ref{thm:boundary_monotone_equidistribution}]
Fix a rational $\phi_0 \in \Diff(Z, \omega)$ and write $[\phi_0] \subset \Diff(Z, \omega)$ for its Hamiltonian isotopy class. Fix any sequence of smooth functions $\{f_i\}_{i \geq 1}$ which are $C^0$-dense in $C^0(Z)$, and assume for the sake of convenience later in the proof that $f_1 \equiv 1$. Fix a sequence $\{\epsilon_N\}_{N \geq 1}$ of positive real numbers, each less than $1$, limiting to $0$ as $N \to \infty$. For any fixed $N \geq 1$, write $\cH_N$ for the space of maps $\psi_0 \in \Diff(Z, \omega)$ which are Hamiltonian isotopic to $\phi_0$ ($\psi_0 \in [\phi_0]$) and have a nondegenerate orbit set $\cO$ such that for each $i \in \{1, \ldots, N\}$, 
$$|\cO(f_i)/|\cO| - A^{-1}\int_Z f_i\,\omega| < \epsilon_N.$$

The proof will proceed in $5$ steps. The first step shows that, for any $N \geq 1$, $\cH_N$ is open in $[\phi_0]$. The next three steps, which constitute the main new part of the proof, show that $\cH_N$ is dense in $[\phi_0]$. The last step finishes off the proof of the theorem. 

\noindent\textbf{Step $1$:} Fix any $N \geq 1$. We observe that by definition, $\cH_N$ is open in $[\phi_0]$ for any $N \geq 1$. If $\psi_0 \in \cH_N$, with $\cO$ a nondegenerate orbit set such that 
$$|\cO(f_i)/|\cO| - A^{-1}\int_Z f_i\,\omega| < \epsilon_N$$
for every $i \in \{1, \ldots, N\}$, then the following is true. Since $\cO$ is nondegenerate, any sufficiently close $\psi_1 \in \cH_N$ will have a nondegenerate orbit set $\cO'$ which is very close to $\cO$. It follows immediately that it will satisfy the same inequality as above. This shows $\cH_N$ is open.  

\noindent\textbf{Step $2$:} The next $3$ steps show $\cH_N$ is dense in $[\phi_0]$. This step performs some required setup. Let $\{\delta_j\}_{j \geq 1}$ be a sequence of small positive constants, which we will specify later. Define two sequences $\{\chi_j\}_{j \geq 1}$  and $\{\eta_j\}_{j \geq 1}$ of cutoff functions as follows. For each $j \geq 1$, $\chi_j$ and $\eta_j$ are smooth functions from $\Sigma$ to $[0,1]$. Also write $g_j := \eta_j(1 - \chi_j)$ for every $j$. We assume the functions $\chi_j$ and $\eta_j$ satisfy the following properties:
\begin{itemize}
\item For each $j \geq 1$, $\chi_j$ is compactly supported in $Z \setminus L$, and $\eta_j \equiv 1$ on an open neighborhood of $Z \subset \Sigma$. 
\item As $j \to \infty$, the restrictions of the functions $\chi_j$ to $Z \setminus L$ converge uniformly on compact subsets of $Z \setminus L$ to the constant $1$. 
\item As $j \to \infty$, the restrictions of the functions $\eta_j$ to $\Sigma \setminus Z$ converge uniformly on compact subsets of $\Sigma \setminus Z$ to the constant $0$.
\item The integral $c_j = \int_\Sigma g_j\,\Omega$ is bounded above by $\delta_j\epsilon_j$ for each $j$. 
\end{itemize}

Let $\{f_i\}_{N \geq 1}$ be the sequence of functions on $Z$ fixed at the start of the proof. For each $i, j \geq 1$ define a smooth function $f_{i, j}$ on $\Sigma$ which is equal to $\chi_j f_i$ on $Z \setminus L$ and $0$ elsewhere. Also define $g_{i,j}$ to be the smooth function which is equal to $g_j f_i$ on $Z$ and $0$ elsewhere. Note that for any $x \in Z$ and any $j$, $f_i(x) = f_{i,j}(x) + g_{i,j}(x)$. Note also that the functions $f_{i,j}$ and the functions $g_j$ (but not necessarily the functions $g_{i,j}$) are locally constant on $L$. 

\noindent\textbf{Step $3$:} Fix any $\psi_0 \in [\phi_0]$, and use Proposition~\ref{prop:extension} to define an extension $\psi \in \Diff(\Sigma, \Omega; L)$. Apply Proposition~\ref{prop:lag_nearly_equidistributed} to the set of $N + 1$ functions
$$\{f_{i,N}\,|\,i=1, \ldots, N\} \cup \{g_N\}$$
with the constant $\epsilon = \delta_N\epsilon_N$. By Lemma \ref{lem:rationality}, $\psi$ is rational. This implies that there is a $C^\infty$-small Hamiltonian $H \in C^\infty([0, 1] \times \Sigma; L)$ and an orbit set $\cO \in \cP_{\bR}(\psi^H)$ such that
\begin{equation} \label{eq:equidistribution1} |\cO(f_{i,N})/|\cO| - B^{-1}\int_{\Sigma} f_{i,N}\,\Omega| < \delta_N\epsilon_N\end{equation}
for any $i \in \{1, \ldots, N\}$ and
\begin{equation} \label{eq:equidistribution2} |\cO(g_N)/|\cO| - B^{-1}c_N| < \delta_N\epsilon_N.\end{equation}

Note in particular that since $f_1 \equiv 1$, 
\begin{equation} \label{eq:equidistribution3}|\cO(\chi_N)/|\cO| - B^{-1}\int_{\Sigma} \chi_N\,\Omega| < \delta_N\epsilon_N.\end{equation}
for each $j \in \{1, \ldots, N\}$. 

\noindent\textbf{Step $4$:} Fix any $i \in \{1, \ldots, N\}$. We now estimate the averages of the function $f_i$ over $\cO^Z$, where $\cO$ is the orbit set from the previous step and $\cO^Z$ is the orbit set constructed by summing the simple orbits from $\cO$ which lie in $Z$. We split up
$$\frac{\cO^Z(f_i)}{|\cO^Z|} = \frac{\cO^Z(f_i)}{|\cO|} \cdot \frac{|\cO|}{|\cO^Z|}$$
and estimate the two terms on the right-hand side separately. To estimate the first term, we split up $f_i = f_{i,N} + g_{i,N}$ and estimate the resulting terms using the inequalities from the previous step:
\begin{equation}
\label{eq:equidistribution4}
\begin{split}
|\frac{\cO^Z(f_i)}{|\cO|} - B^{-1}\int_\Sigma f_i\,\Omega| &\leq |\frac{\cO(f_{i,N})}{|\cO|} - B^{-1}\int_\Sigma f_i\,\Omega| + \frac{\cO^Z( g_{i,N})}{|\cO|} \\
&\leq B^{-1}\int_\Sigma g_{i,N}\,\Omega + \delta_N\epsilon_N + \frac{\cO^Z(g_{i,N})}{|\cO|} \\
&\leq B^{-1}\int_\Sigma g_{i,N}\,\Omega + \delta_N\epsilon_N + \|f_i\|_{C^0}(B^{-1}c_N + \delta_N\epsilon_N) \\
&\leq 2(B^{-1}\|f_i\|_{C^0}c_N + \delta_N\epsilon_N) \\
&\leq 2(B^{-1}\|f_i\|_{C^0} + 1)\delta_N\epsilon_N.
\end{split}
\end{equation}

The first inequality is a consequence of the triangle inequality and the identity $\cO^Z(f_i) = \cO(f_{i,N}) + \cO^Z(g_{i,N})$. The second inequality uses \eqref{eq:equidistribution1} to bound the difference between the average of $f_{i,N}$ over $\cO$ and the average of $f_i$ over $\Sigma$. The third inequality uses the bound $\cO^Z(g_{i,N}) \leq \|f_i\|_{C^0}\cO(g_N)$ and \eqref{eq:equidistribution2}. 

Next, we estimate the quotient $|\cO^Z|/|\cO|$. This proceeds in a similar manner to \eqref{eq:equidistribution4}, except now we use the fact that $1 = \chi_N + g_N$ on $Z$ and the inequality \eqref{eq:equidistribution3}:
\begin{equation}
\label{eq:equidistribution5}
\begin{split}
|\frac{|\cO^Z|}{|\cO|} - \frac{A}{B}| &\leq |\frac{\cO(\chi_N)}{|\cO|} - \frac{A}{B}| + \frac{\cO^Z(g_N)}{|\cO|} \\
&\leq B^{-1}c_N + \delta_N\epsilon_N + \frac{\cO^Z(g_N)}{|\cO|} \\
&\leq 2(B^{-1}c_N + \delta_N\epsilon_N) \\
&\leq 2(B^{-1} + 1)\delta_N\epsilon_N
\end{split}
\end{equation}

The first inequality is a consequence of the triangle inequality and the identity $|\cO^Z| = \cO(\chi_N) + \cO^Z(g_N)$. The second inequality uses \eqref{eq:equidistribution3} and the bound
$$|1 - \int_\Sigma \chi_N\,\Omega| = \int_Z (1 - \chi_N)\,\omega \leq c_N.$$
The third inequality uses \eqref{eq:equidistribution2} along with the fact that $\cO^Z(g_N) \leq \cO(g_N)$. Now fix the scaling constant $\delta_N$ so that $\delta_N\epsilon_N \leq \frac{A}{100B}$. Then \eqref{eq:equidistribution5} implies 
\begin{equation}
\label{eq:equidistribution6}
\begin{split}
|\frac{|\cO|}{|\cO^Z|} - \frac{B}{A}| &= \big(\frac{|\cO^Z|}{|\cO|} \cdot \frac{A}{B}\big)^{-1}|\frac{|\cO^Z|}{|\cO|} - \frac{A}{B}| \\
&\leq C(A, B)\delta_N\epsilon_N
\end{split}
\end{equation}
for some constant $C(A, B) \geq 1$ depending only on $A$ and $B$.

Denote by $E_1$ and $E_2$ the terms on the left-hand sides of \eqref{eq:equidistribution4} and \eqref{eq:equidistribution6}, respectively. Combine \eqref{eq:equidistribution4} and \eqref{eq:equidistribution6} to deduce the following bound:
\begin{equation*}
\begin{split}
|\frac{\cO^Z(f_i)}{|\cO^Z|} - A^{-1}\int_Z f_i\,\omega| &\leq E_1 \cdot E_2 + E_2 \cdot B^{-1}\int_Z f_i\,\omega + E_1 \cdot BA^{-1} \\
&\leq C(A, B, i) \delta_N \epsilon_N.
\end{split}
\end{equation*}

The first inequality is a consequence of the triangle inequality. The second inequality plugs in \eqref{eq:equidistribution4} and \eqref{eq:equidistribution6}. Here $C(A, B, i) \geq 1$ is a constant depending only on $A$, $B$, and the $C^0$ norm of $f_i$. Fix the scaling constant $\delta_N$ to be smaller than $C(A, B, i)^{-1}$ for every $i \in \{1, \ldots, N\}$ and conclude that   
$$|\frac{\cO^Z(f_i)}{|\cO^Z|} - \int_Z f_i\,\omega| < \epsilon_N$$
for all $i \in \{1, \ldots, N\}$. Recall that $\cO$ is an orbit set of a $C^\infty$-small Hamiltonian perturbation $\psi^H$ of the extension $\psi$ of the original $\psi_0 \in \Diff(Z)$, and so $\cO^Z$ is an orbit set of its restriction $\psi^H_0$. After applying another small Hamiltonian perturbation via Lemma~\ref{lem:make_nondegenerate} to make $\cO^Z$ nondegenerate, we conclude that $\psi_0^H \in \cH_N$ and so $\cH_N$ is dense in $[\phi_0]$. 

\noindent\textbf{Step $5$:} Write 
$$\cH_{\text{good}} := \bigcap_{N \geq 1} \cH_N$$ 
for the intersection of the sets $\cH_N$. Since each set is open and dense, $\cH_{\text{good}}$ is residual inside $[\phi_0]$. It remains to prove that any $\psi_0 \in \cH_{\text{good}}$ has an equidistributed sequence of orbit sets. Fix $\psi_0 \in \cH_{\text{good}}$. Then by definition, $\psi_0$ has a sequence of orbit sets $\{\cO_N\}_{N \geq 1}$ such that for each $N$ and each $i \in \{1, \ldots, N\}$, 
$$|\frac{\cO_N(f_i)}{|\cO_N|} - \int_Z f_i\,\omega| < \epsilon_N.$$

Fix any continuous function $f$ on $Z$. Write $c_N = \sup_{i \in \{1, \ldots, N\}} \|f - f_i\|_{C^0}$ for every $N \geq 1$. Note that $c_N \to 0$ as $N \to \infty$ since the sequence $\{f_i\}_{i \geq 1}$ is dense in $C^0(Z)$. Then
\begin{align*}
\limsup_{N \to \infty} |\frac{\cO_N(f)}{|\cO_N|} - \int_Z f\,\omega| &\leq \limsup_{N \to \infty}\sup_{i \in \{1, \ldots, N\}}(|\frac{\cO_N(f_i)}{|\cO_N|} - \int_Z f_i\,\omega| + |\frac{\cO_N(f - f_i)}{|\cO_N|} - \int_Z (f - f_i)\,\omega|) \\
&\leq \limsup_{N \to \infty} \epsilon_N + 2c_N=0.
\end{align*}

This implies that the sequence $\{\cO_N\}_{N \geq 1}$ equidistributes, which proves the theorem. 
\end{proof}

We conclude by proving Theorem~\ref{thm:boundary_equidistribution}. The proof is a minor modification of the proof of Theorem~\ref{thm:boundary_monotone_equidistribution}. 

\begin{proof}[Proof of Theorem~\ref{thm:boundary_equidistribution}]
Fix any sequence of smooth functions $\{f_i\}_{i \geq 1}$ which are $C^0$-dense in $C^0(Z)$ with $f_1 \equiv 1$. Fix a sequence $\{\epsilon_N\}_{N \geq 1}$ of positive real numbers, each less than $1$, limiting to $0$ as $N \to \infty$. For any fixed $N \geq 1$, write $\cD_N$ for the space of maps $\psi_0 \in \Diff(Z, \omega)$ which have a nondegenerate orbit set $\cO$ such that for each $i \in \{1, \ldots, N\}$, 
$$|\cO(f_i)/|\cO| - \int_Z f_i\,\omega| < \epsilon_N.$$

The same argument as in Step $1$ of the proof of Theorem~\ref{thm:boundary_monotone_equidistribution} shows that $\cD_N$ is open for every $N \geq 1$. By Proposition~\ref{prop:rational_maps_dense}, rational maps are dense in $\Diff(Z, \omega)$. Use the argument in Steps $2$--$4$ of the proof of Theorem~\ref{thm:boundary_monotone_equidistribution} to show that $\cD_N$ is dense. Then the argument in Step $5$ of the proof of Theorem~\ref{thm:boundary_monotone_equidistribution}, repeated verbatim, shows that any map in the residual set
$$\cD_{\text{good}} := \bigcap_{N \geq 1} \cD_N \subset \Diff(Z, \omega)$$
has an equidistributed sequence of orbit sets, which finishes off the proof of the theorem. 
\end{proof}

\bibliographystyle{alpha}
\bibliography{density}

\begin{thebibliography}{CGHHL21}

\bibitem[AI16]{asaokairie}
Masayuki Asaoka and Kei Irie.
\newblock A {$C^\infty$--}closing lemma for {H}amiltonian diffeomorphisms of
  closed surfaces.
\newblock {\em Geometric and Functional Analysis}, 26(5):1245--1254, 2016.

\bibitem[BSHSa21]{action_linking}
David Bechara~Senior, Umberto~L. Hryniewicz, and Pedro A.~S. Salom\~{a}o.
\newblock On the relation between action and linking.
\newblock {\em J. Mod. Dyn.}, 17:319--336, 2021.

\bibitem[CGHHL21]{cristofarogardiner2021contact}
Dan Cristofaro-Gardiner, Umberto Hryniewicz, Michael Hutchings, and Hui Liu.
\newblock {Contact three-manifolds with exactly two simple Reeb orbits}.
\newblock {\em arXiv preprint arXiv:2102.04970}, 2021.

\bibitem[CGPZ21]{CGPZ21}
Dan Cristofaro-Gardiner, Rohil Prasad, and Boyu Zhang.
\newblock Periodic {F}loer homology and the smooth closing lemma for
  area-preserving surface diffeomorphisms.
\newblock {\em arXiv preprint arXiv:2110.02925}, 2021.

\bibitem[EH21]{edtmairHutchings}
Oliver Edtmair and Michael Hutchings.
\newblock {PFH} spectral invariants and {$C^\infty$} closing lemmas.
\newblock {\em arXiv preprint arXiv:2110.02463}, 2021.

\bibitem[EPS23]{EPS23}
Alberto Enciso and Daniel Peralta-Salas.
\newblock Obstructions to topological relaxation for generic magnetic fields.
\newblock {\em arXiv preprint arXiv:2308.15383}, 2023.

\bibitem[Hum21]{Humiliere21}
Vincent Humili\`ere.
\newblock Expos\'{e} {B}ourbaki {$1168$}: A {$C^\infty$} closing lemma (after
  asaoka and irie).
\newblock {\em Ast\'{e}risque}, (430):55--92, 2021.

\bibitem[Hut16]{hutchings_mean_action}
Michael Hutchings.
\newblock Mean action and the {C}alabi invariant.
\newblock {\em J. Mod. Dyn.}, 10:511--539, 2016.

\bibitem[LC22]{Cal}
Patrice Le~Calvez.
\newblock {A finite dimensional proof of a result of Hutchings about irrational
  pseudo-rotations}.
\newblock {\em arXiv preprint arxiv.2207.07319}, 2022.

\bibitem[Mat82]{mather}
John~N. Mather.
\newblock Existence of quasiperiodic orbits for twist homeomorphisms of the
  annulus.
\newblock {\em Topology}, 21(4):457--467, 1982.

\bibitem[Pir21]{Pirnapasov21}
Abror Pirnapasov.
\newblock Hutchings' inequality for the {C}alabi invariant revisited with an
  application to pseudo-rotations.
\newblock {\em arXiv preprint arXiv:2102.09533}, 2021.

\bibitem[Pra21]{equidistribution}
Rohil Prasad.
\newblock Generic equidistribution of periodic orbits for area-preserving
  surface maps.
\newblock {\em arXiv preprint arXiv:2112.14601}, 2021.

\bibitem[Wei21]{weiler_mean_action}
Morgan Weiler.
\newblock Mean action of periodic orbits of area-preserving annulus
  diffeomorphisms.
\newblock {\em J. Topol. Anal.}, 13(4):1013--1074, 2021.

\end{thebibliography}

\end{document}